\pdfoutput=1
\RequirePackage{ifpdf}
\ifpdf 
\documentclass[pdftex]{sigma}
\else
\documentclass{sigma}
\fi

\usepackage{tikz}
\numberwithin{equation}{section}

\newtheorem{Theorem}{Theorem}[section]
\newtheorem{Corollary}{Corollary}[section]
\newtheorem{Lemma}{Lemma}[section]
\newtheorem{Proposition}{Proposition}[section]
{ \theoremstyle{definition}
\newtheorem{Definition}{Definition}[section]
\newtheorem{Example}{Example}[section]
\newtheorem{Remark}{Remark}[section] }

\newcommand{\N}{\mathbb{N}}

\newcommand{\K}{\mathbb{K}}
\newcommand{\I}{\mathcal{I}}
\newcommand{\ospq}{\mathfrak{osp}_q(1\vert 2)}
\newcommand{\Uqsl}{U_q(\mathfrak{sl}_2)}

\begin{document}
\allowdisplaybreaks

\newcommand{\arXivNumber}{1908.11654}

\renewcommand{\PaperNumber}{099}

\FirstPageHeading

\ShortArticleName{Higher Rank Relations for the Askey--Wilson and $q$-Bannai--Ito Algebra}

\ArticleName{Higher Rank Relations for the Askey--Wilson\\ and $\boldsymbol{q}$-Bannai--Ito Algebra}

\Author{Hadewijch DE CLERCQ}

\AuthorNameForHeading{H.~De Clercq}

\Address{Department of Electronics and Information Systems, Faculty of Engineering and Architecture,\\ Ghent University, Belgium}
\Email{\href{mailto:hadewijch.declercq@ugent.be}{hadewijch.declercq@ugent.be}}

\ArticleDates{Received September 03, 2019, in final form December 13, 2019; Published online December 19, 2019}

\Abstract{The higher rank Askey--Wilson algebra was recently constructed in the $n$-fold tensor product of $U_q(\mathfrak{sl}_2)$. In this paper we prove a class of identities inside this algebra, which generalize the defining relations of the rank one Askey--Wilson algebra. We extend the known construction algorithm by several equivalent methods, using a novel coaction. These allow to simplify calculations significantly. At the same time, this provides a proof of the corresponding relations for the higher rank $q$-Bannai--Ito algebra.}


\Classification{16T05; 16T15; 17B37; 81R50}

\section{Introduction}

The Askey--Wilson algebra was introduced in \cite{Zhedanov-1991} as an algebraic foundation for the bispectral problem of the Askey--Wilson orthogonal polynomials \cite{Koekoek&Lesky&Swarttouw-2010}. More precisely, Zhedanov defined this algebra by generators and relations, which turned out to be satisfied when realizing the generators as the Askey--Wilson $q$-difference operator on the one hand, and the multiplication with the variable on the other. A central extension, which allows a $\mathbb{Z}_3$-symmetric presentation, was defined and studied by Terwilliger \cite{Terwilliger-2011}. He calls this central extension the universal Askey--Wilson algebra. We will denote it by ${\rm AW}(3)$.

The irreducible finite-dimensional representations of the Askey--Wilson algebra have been classified by Leonard pairs \cite{Terwilliger-2001,Terwilliger&Vidunas-2004} and a similar classification for the universal Askey--Wilson algebra ${\rm AW}(3)$ has appeared in \cite{Huang-2015}.
Further applications arise in the theory of special functions~\cite{Baseilhac&Martin&Vinet&Zhedanov-2019, Koornwinder&Mazzocco-2018}, tridiagonal and Leonard pairs \cite{Nomura-2005,Terwilliger-2004, Vidunas-2008}, superintegrable quantum systems \cite{Baseilhac-2005} and the reflection equation \cite{Baseilhac-2005-2}. Its use to quantum mechanics is further emphasized by the identification of the Askey--Wilson algebra as a quotient of the $q$-Onsager algebra \cite{Terwilliger-2004-2}, which originates from statistical mechanics. This connection was later extended to the universal Askey--Wilson algebra \cite{Terwilliger-2018}. Furthermore, an explicit homomorphism from the original Askey--Wilson algebra to the double affine Hecke algebra of type $\big(C_1^{\vee},C_1\big)$ has been constructed in \cite{Ito&Terwilliger-2010, Koornwinder-2007}. Recently, connections with $q$-Higgs algebras \cite{Frappat&Gaboriaud&Ragoucy&Vinet-2019} and Howe dual pairs \cite{Frappat&Gaboriaud&Ragoucy&Vinet-2019-2} have been obtained.

The original Askey--Wilson algebra \cite{Zhedanov-1991} was realized inside the quantum group $U_q(\mathfrak{sl}_2)$ in \cite{Granovskii&Zhedanov-1993}. A different realization in the three-fold tensor product of $U_q(\mathfrak{sl}_2)$ was given in \cite{Granovskii&Zhedanov-1993-2}, and was later extended to the universal Askey--Wilson algebra ${\rm AW}(3)$ of \cite{Terwilliger-2011} in \cite{Huang-2017}. In the latter reference, ${\rm AW}(3)$ is embedded in the threefold tensor product of $\Uqsl$: if one denotes by $\Lambda$ the quadratic Casimir element of $\Uqsl$ and by $\Delta$ its coproduct, and defines
\begin{gather}
\Lambda_{\{1\}} = \Lambda\otimes 1\otimes 1, \qquad \Lambda_{\{2\}} = 1\otimes\Lambda\otimes 1, \qquad \Lambda_{\{3\}} = 1\otimes 1\otimes\Lambda,\nonumber\\
\Lambda_{\{1,2\}} = \Delta(\Lambda)\otimes 1, \qquad \Lambda_{\{2,3\}} = 1\otimes\Delta(\Lambda), \qquad
\Lambda_{\{1,2,3\}} = (1\otimes\Delta)\Delta(\Lambda),\label{Lambdas without 13}
\end{gather}
then these elements generate the universal Askey--Wilson algebra. Indeed, they satisfy the $q$-commutation relations
\begin{gather}
\label{rank one 12-23}
[\Lambda_{\{1,2\}},\Lambda_{\{2,3\}}]_q = \big(q^{-2}-q^2\big) \Lambda_{\{1,3\}} + \big(q-q^{-1}\big)\big(\Lambda_{\{2\}}\Lambda_{\{1,2,3\}} + \Lambda_{\{1\}}\Lambda_{\{3\}}\big), \\
\label{rank one 23-13}
[\Lambda_{\{2,3\}},\Lambda_{\{1,3\}}]_q = \big(q^{-2}-q^2\big)\Lambda_{\{1,2\}} + \big(q-q^{-1}\big)\big(\Lambda_{\{3\}}\Lambda_{\{1,2,3\}} + \Lambda_{\{1\}}\Lambda_{\{2\}}\big), \\
\label{rank one 13-12}
[\Lambda_{\{1,3\}},\Lambda_{\{1,2\}}]_q = \big(q^{-2}-q^2\big)\Lambda_{\{2,3\}} + \big(q-q^{-1}\big)\big(\Lambda_{\{1\}}\Lambda_{\{1,2,3\}} + \Lambda_{\{2\}}\Lambda_{\{3\}}\big),
\end{gather}
where $\Lambda_{\{1,3\}}$ is defined through relation (\ref{rank one 12-23}), and $\Lambda_{\{1,2,3\}}$ and all $\Lambda_{\{i\}}$ are central. This coincides with the presentation for ${\rm AW}(3)$ given in~\cite{Terwilliger-2011}. This illustrates that there exists an algebra homomorphism from ${\rm AW}(3)$ to $\Uqsl^{\otimes 3}$ and this map turns out to be injective, as shown in \cite[Theorem~4.8]{Huang-2017}.

In \cite{DeBie&DeClercq&vandeVijver-2018} this approach was generalized to $n$-fold tensor products for arbitrary $n$, which leads to an extension of the universal Askey--Wilson algebra to higher rank, denoted ${\rm AW}(n)$. The same algorithm allows to construct a higher rank extension of the $q$-Bannai--Ito algebra, which is isomorphic to ${\rm AW}(3)$ under a transformation $q\to -q^2$ and allows a similar embedding in $\ospq^{\otimes 3}$ \cite{Genest&Vinet&Zhedanov-2016}. Also the limiting cases $q = 1$ provide interesting algebras, as summarized graphically below.
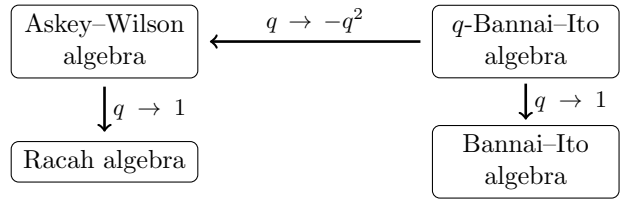
\begin{figure}[h]
\centering
\begin{tikzpicture}[scale = 0.8, every node/.style={scale=0.9}]
\node[draw = black, rectangle, rounded corners = 0.10 cm, text width = 2.5 cm, text centered] at (0,0) {Askey--Wilson\\ algebra};
\node[draw = black, rectangle, rounded corners = 0.10 cm, text width = 2.5 cm, text centered] at (0,-2) {Racah algebra};
\node[draw = black, rectangle, rounded corners = 0.10 cm, text width = 2.5 cm, text centered] at (7,0) {$q$-Bannai--Ito algebra};
\node[draw = black, rectangle, rounded corners = 0.10 cm, text width = 2.5 cm, text centered] at (7,-2) {Bannai--Ito algebra};
\draw[->, line width = 0.35 mm] (0,-0.75) -- (0,-1.5);
\draw[<-, line width = 0.35 mm] (1.75,0) -- (5.25,0);
\draw[->, line width = 0.35 mm] (7,-0.7) -- (7,-1.3);
\node[draw = none, text width = 2.5 cm, text centered] at (0.75, -1.125) {\small $q\to 1$};
\node[draw = none, text width = 2.5 cm, text centered] at (3.5, 0.35) {\small $q\to -q^2 $};
\node[draw = none, text width = 2.5 cm, text centered] at (7.75, -1) {\small $q\to 1$};
\end{tikzpicture}
\end{figure}

Such higher rank algebras are motivated by their role as symmetry algebras for superintegrable quantum systems of higher dimension. This has been confirmed in the limiting case $q = 1$ \cite{DeBie&Genest&Lemay&Vinet-2017, DeBie&Genest&vandeVijver&Vinet-2016, DeBie&Genest&Vinet-2016}, and later also for general $q$ \cite{DeBie&DeClercq&vandeVijver-2018}. In both cases, the Hamiltonians under consideration are built from Dunkl operators with $\mathbb{Z}_2^n$ symmetry \cite{Dunkl-1989}, possibly $q$-deformed \cite{Cherednik-1995}. Moreover, these higher rank algebras allow to extend known connections with orthogonal polynomials to multiple variables. This was achieved in \cite{DeBie&DeClercq-2019} for the $q$-Bannai--Ito algebra. To be concrete, an action of the higher rank $q$-Bannai--Ito algebra on an abstract vector space was considered, leading to various orthonormal bases for this space. The overlap coefficients between such bases turned out to be multivariable $(-q)$-Racah polynomials, the truncated analogs of Askey--Wilson polynomials. This has allowed to construct a realization of the higher rank $q$-Bannai--Ito algebra with Iliev's $q$-difference operators \cite{Iliev-2011}, which have thereby obtained an algebraic interpretation.

The construction of ${\rm AW}(n)$ is rather intricate: in \cite{DeBie&DeClercq&vandeVijver-2018} we have outlined an algorithm which repeatedly applies the $\Uqsl$-coproduct $\Delta$ and a coaction $\tau$ to the Casimir element $\Lambda$, in a~specific order. This way we construct, as an extension of (\ref{rank one 12-23})--(\ref{rank one 13-12}), an element $\Lambda_A\in\Uqsl^{\otimes n}$ for each $A\subseteq \{1,\dots,n\}$. In this paper we rephrase this extension algorithm in more accessible notation and provide alternative construction methods for the elements~$\Lambda_A$ which use a novel coaction. This new approach is of major use to derive algebraic identities in ${\rm AW}(n)$, as we showcase in Theorems~\ref{thm - commutation general} and~\ref{thm - most general algebra relations} by significantly generalizing the algebraic relations given in~\cite{DeBie&DeClercq&vandeVijver-2018}. The main achievement of this paper is hence a general criterion for two generators $\Lambda_A$ and $\Lambda_B$ of ${\rm AW}(n)$ to commute or to satisfy a relation of the form
\begin{gather}\label{general relation}
[\Lambda_A,\Lambda_B]_q = \big(q^{-2}-q^2\big)\Lambda_{(A\cup B)\setminus(A\cap B)} + \big(q-q^{-1}\big)\big(\Lambda_{A\cap B}\Lambda_{A\cup B} + \Lambda_{A\setminus(A\cap B)}\Lambda_{B\setminus(A\cap B)}\big).\!\!\!
\end{gather}
More precisely, we show in Theorem \ref{thm - commutation general} that
\[
[\Lambda_A,\Lambda_B] = 0 \qquad \mathrm{if}\quad B\subseteq A,
\]
and in Theorem \ref{thm - most general algebra relations} we prove that the relation (\ref{general relation}) is satisfied for
\begin{alignat}{3}
& A= A_1 \cup A_2 \cup A_4, \qquad && B = A_2 \cup A_3, &\nonumber\\
& A= A_2 \cup A_3, \qquad && B = A_1 \cup A_3 \cup A_4,& \nonumber\\
& A= A_1 \cup A_3 \cup A_4, \qquad && B = A_1 \cup A_2 \cup A_4,&\label{condition on sets}
\end{alignat}
where $A_1, A_2, A_3, A_4 \subseteq\{1,\dots,n\}$ are such that for each $i\in\{1,2,3\}$ one has either $\max(A_i) < \min(A_{i+1})$ or $A_i = \varnothing$ or $A_{i+1} = \varnothing$.

It is not clear at this point whether these relations define the algebra ${\rm AW}(n)$ abstractly, or, in case the answer turns to be negative, which supplementary relations should be added in order to attain this purpose. However, calculations with computer algebra packages suggest that the condition (\ref{condition on sets}) describes the most general situation for the relations~(\ref{general relation}) to be satisfied.

Our methods are elementary and intrinsic: they are independent of the expressions for the coactions and the $\Uqsl$-Casimir $\Lambda$, and only recur to natural algebraic properties like coassociativity and the cotensor product property. As a consequence, the results of this paper are equally applicable to the higher rank $q$-Bannai--Ito algebra of \cite{DeBie&DeClercq&vandeVijver-2018}, without any modification.

The paper is organized as follows. In Section \ref{Section - Defining the higher rank generators} we construct the higher rank Askey--Wilson algebra ${\rm AW}(n)$ as a subalgebra of $\Uqsl^{\otimes n}$ through different extension processes, which we prove to be equivalent. Section \ref{Section - Main results} lists the main results of this paper and explains the general strategy of proof. Consequently, in Sections \ref{section - proof fundamental relations} and \ref{section - more commutation relations} we prove some intermediate results which will be relied on in Sections \ref{section - proof commutation general} and \ref{section - proof most general algebra relations}, where we prove Theorems \ref{thm - commutation general} and \ref{thm - most general algebra relations}. Finally in Section~\ref{section q-BI} we introduce similar extension processes to construct a~higher rank $q$-Bannai--Ito algebra as a~subalgebra of $\ospq^{\otimes n}$. We state two explicit theorems describing their algebraic relations.

\section{Defining the higher rank generators}\label{Section - Defining the higher rank generators}

Throughout this paper, we will work with the quantum group $\Uqsl$, which can be presented as the associative algebra over a field $\K$ with generators $E$, $F$, $K$ and $K^{-1}$, and relations
\[
KK^{-1} = K^{-1}K = 1, \qquad KE = q^2 EK, \qquad KF = q^{-2} FK, \qquad [E,F] = \frac{K-K^{-1}}{q-q^{-1}},
\]
where $q$ is a fixed parameter in $\mathbb{K}$, assumed not to be a root of unity. A Casimir element, which commutes with all elements of $\Uqsl$, is given by
\begin{gather}\label{Lambda def}
\Lambda = \big(q-q^{-1}\big)^2 EF + q^{-1} K + qK^{-1}.
\end{gather}

The quantum group $\Uqsl$ has the structure of a bialgebra: it is equipped with a coproduct $\Delta\colon \Uqsl\to\Uqsl^{\otimes 2}$, which satisfies the coassociativity property $(1\otimes\Delta)\Delta = (\Delta\otimes 1)\Delta$, and a~counit $\epsilon\colon \Uqsl\to\mathbb{K}$ satisfying $(1\otimes\epsilon)\Delta = (\epsilon\otimes 1)\Delta = 1$, where $1$ denotes the identity mapping on $\Uqsl$. Explicitly, they are given by
\begin{gather}\label{Coproduct}
\Delta(E) = E\otimes 1 + K\otimes E, \qquad\! \Delta(F) = F\otimes K^{-1} + 1\otimes F, \qquad\! \Delta\big(K^{\pm 1}\big) = K^{\pm 1}\otimes K^{\pm 1}, \!\! \\
\epsilon(E) = \epsilon(F) = 0, \qquad \epsilon(K) = \epsilon\big(K^{-1}\big) = 1.\nonumber
\end{gather}

A binary operation we will often use is the so-called $q$-commutator. For $X,Y\in\Uqsl$ we write
\[
[X,Y]_q = qXY - q^{-1}YX.
\]

For $i$ and $j$ natural numbers with $i\leq j$, we will write $[i;j]$ to denote the discrete interval $\{i,i+1,\dots,j-1,j\}$. If we consider disjoint unions of discrete intervals, often denoted by $[i_1;j_1]\cup[i_2;j_2]\cup\dots\cup[i_k;j_k]$, it is always understood that $i_{\ell}\leq j_{\ell}<i_{\ell+1}-1$ for all $\ell$. Note that this implies that $i_k \geq i_1+2k-2$. Moreover, if $B$ is any set of natural numbers and $a\in\N$, we will write $B - a$ for the set $\{b-a\colon b\in B\}$.

\subsection{Coideals and comodules}

In \cite{DeBie&DeClercq&vandeVijver-2018} we have introduced the following $\Uqsl$-subalgebra.

\begin{Definition}We denote by $\I_R$ the subalgebra of $\Uqsl$ generated by $EK^{-1}$, $F$, $K^{-1}$ and $\Lambda$, and we define the algebra morphism $\tau_R\colon \I_R\to \Uqsl\otimes\I_R$ through its action on the generators:
\begin{gather*}
\tau_R\big(EK^{-1}\big) = K^{-1}\otimes EK^{-1}, \\
\tau_R(F) = K\otimes F - q^{-3}\big(q-q^{-1}\big)^2 F^2K \otimes EK^{-1} + q^{-1}\big(q+q^{-1}\big) FK \otimes K^{-1} - q^{-1} FK\otimes \Lambda, \\
\tau_R\big(K^{-1}\big) = 1\otimes K^{-1}-q^{-1}\big(q-q^{-1}\big)^2 F\otimes EK^{-1}, \\
\tau_R(\Lambda) = 1\otimes \Lambda.
\end{gather*}
\end{Definition}

It is readily checked that these definitions comply with the algebra relations in $\mathcal{I}_R$. The following important observation was made in \cite[Proposition~3]{DeBie&DeClercq&vandeVijver-2018}.

\begin{Proposition}The algebra $\mathcal{I}_R$ is a left coideal subalgebra of $U_q(\mathfrak{sl}_2)$ and a left $U_q(\mathfrak{sl}_2)$-comodule with coaction $\tau_R$. This means that $\Delta(\mathcal{I}_R) \subset \Uqsl\otimes\mathcal{I}_R$ and that one has
\begin{gather}\label{R comodule prop 1}
(1\otimes\tau_R)\tau_R = (\Delta\otimes 1)\tau_R, \\
(\epsilon\otimes 1)\tau_R = 1.\nonumber
\end{gather}
\end{Proposition}

An interpretation of the coaction $\tau_R$ in terms of the universal $R$-matrix for $\Uqsl$ was recently given in \cite{Crampe&Vinet&Zaimi-2019}. This coaction was constructed so as to satisfy the identity
\begin{gather}
\label{Lambda 13 with tau_R}
\Lambda_{\{1,3\}} = (1\otimes\tau_R)\Delta(\Lambda),
\end{gather}
with $\Lambda_{\{1,3\}}$ defined through (\ref{Lambdas without 13}) and (\ref{rank one 12-23}). A similar mapping $\tau_L$ can be constructed by demanding that
\begin{gather}
\label{Lambda 13 with tau_L}
\Lambda_{\{1,3\}} = (\tau_L\otimes 1)\Delta(\Lambda).
\end{gather}
This suggests the following definition.

\begin{Definition}\label{I_L and tau_L def}We denote by $\I_L$ the subalgebra of $\Uqsl$ generated by $E$, $FK$, $K$ and $\Lambda$, and we define the algebra morphism $\tau_L\colon \I_L\to \I_L\otimes\Uqsl$ through its action on the generators:
\begin{gather*}
\tau_L(E) = E\otimes K, \\
\tau_L(FK) = FK\otimes K^{-1} - q^{-1}\big(q-q^{-1}\big)^2E\otimes F^2K + q\big(q+q^{-1}\big)K\otimes F - q\Lambda\otimes F, \\
\tau_L(K) = K\otimes 1 - q^{-1}\big(q-q^{-1}\big)^2E\otimes FK, \\
\tau_L(\Lambda) = \Lambda\otimes 1.
\end{gather*}
\end{Definition}

This subalgebra behaves in a similar fashion.

\begin{Proposition}The algebra $\mathcal{I}_L$ is a right coideal subalgebra of $U_q(\mathfrak{sl}_2)$ and a right $U_q(\mathfrak{sl}_2)$-comodule with coaction $\tau_L$.
\end{Proposition}
\begin{proof}It suffices to check explicitly on each of the generators that
\begin{enumerate}\itemsep=0pt
\item[1)] $\I_L$ is a right coideal of $\Uqsl$: $\Delta(\mathcal{I}_L) \subset \I_L\otimes\Uqsl$,
\item[2)] $\tau_L$ is a right coaction: it preserves the algebra relations in $\mathcal{I}_L$ and satisfies
\begin{gather}\label{L comodule prop 1}
(\tau_L\otimes 1)\tau_L = (1\otimes \Delta)\tau_L, \\
(1\otimes\epsilon)\tau_L = 1.\tag*{\qed}
\end{gather}
\end{enumerate}\renewcommand{\qed}{}
\end{proof}

It is readily checked that the element $\Delta(\Lambda)$ lies in $\mathcal{I}_L\otimes\mathcal{I}_R$. Indeed, it follows immediately from (\ref{Lambda def}) and (\ref{Coproduct}) that one has
\[
\Delta(\Lambda) = \Lambda\otimes K^{-1} + K\otimes\Lambda - \big(q+q^{-1}\big) K\otimes K^{-1} + \big(q-q^{-1}\big)^2 \big(E\otimes F+q^{-2} FK\otimes EK^{-1}\big).
\]

In the language of category theory \cite{Abrams&Weibel-2002}, the equality of (\ref{Lambda 13 with tau_R}) and (\ref{Lambda 13 with tau_L}) can be phrased as follows.

\begin{Corollary}\label{cor cotensor product}
The element $\Delta(\Lambda) \in\I_L\otimes\I_R$, with $\Lambda$ defined in~\eqref{Lambda def}, belongs to the cotensor product of the coideal comodule subalgebras $\I_L$ and $\I_R$ of $\Uqsl$:
\begin{gather}\label{cotensor product property}
(1\otimes\tau_R)\Delta(\Lambda) = (\tau_L\otimes 1)\Delta(\Lambda).
\end{gather}
\end{Corollary}

\subsection{The right extension process} Our goal in this section will be to associate to each set $A\subseteq [1;n]$ an element $\Lambda_A\in\Uqsl^{\otimes n}$, which will serve as a generator for the higher rank Askey--Wilson algebra ${\rm AW}(n)$. For the empty set, this will simply be the scalar
\begin{gather}
\label{Lambda emptyset}
\Lambda_{\varnothing} = q+q^{-1}.
\end{gather}
For general $A$, a construction algorithm was given in \cite{DeBie&DeClercq&vandeVijver-2018}. We will repeat it here in a more accessible notation.

\begin{Definition}
\label{right extension process def}
For any set $A = \{a_1,\dots,a_m\}\subseteq [1;n]$, ordered such that $a_i < a_{i+1}$ for all $i$, we define $\Lambda_A\in\Uqsl^{\otimes n}$ by
\[
\Lambda_A = 1^{\otimes (a_1-1)}\otimes\left(\overrightarrow{\prod_{i = 2}^m}\, \mu^{A}_i\right)(\Lambda) \otimes 1^{\otimes (n-a_m)},
\]
with
\begin{gather}\label{mu def}
\mu_i^{A} = \left(\overrightarrow{\prod_{\ell = a_{i-1}-a_1+1}^{a_i-a_1-1}}\big(1^{\otimes\ell}\otimes\tau_R\big)\right)\big(1^{\otimes (a_{i-1}-a_1)}\otimes\Delta\big),
\end{gather}
where it is understood that the term between brackets in (\ref{mu def}) is absent if $a_{i} = a_{i-1}+1$.
\end{Definition}
\begin{Example}\label{Example right}
If $n = 9$ and $A = \{2,4,5,8\}$, then we have
\[
\Lambda_{\{2,4,5,8\}} = 1 \otimes \big(\mu_4^{\{2,4,5,8\}}\mu_3^{\{2,4,5,8\}}\mu_2^{\{2,4,5,8\}}\big)(\Lambda)\otimes 1,
\]
with $\mu_2^{\{2,4,5,8\}} = (1\otimes\tau_R)\Delta$, $\mu_3^{\{2,4,5,8\}} = 1^{\otimes 2}\otimes\Delta$ and $\mu_4^{\{2,4,5,8\}} = \big(1^{\otimes 5}\otimes\tau_R\big) \big(1^{\otimes 4}\otimes\tau_R\big)\big(1^{\otimes 3}\otimes\Delta\big)$.
\end{Example}

\looseness=-1 The rationale behind this construction is that each element of $A$, except for its minimum $a_1$, corresponds to an application of $\Delta$, whereas the coaction $\tau_R$ is used to \emph{create the gaps} between the elements of $A$. In the example above, $\tau_R$ is applied first once and then twice, corresponding to a \emph{hole} of size 1 between 2 and 4 and one of size 2 between 5 and 8. The improvement with respect to the notation of~\cite{DeBie&DeClercq&vandeVijver-2018} lies in the fact that here we iterate over the elements of the set~$A$ rather than over all elements of~$[1;n]$, such that we can avoid distinguishing several cases as in~\cite{DeBie&DeClercq&vandeVijver-2018}.

\begin{Definition}
The Askey--Wilson algebra of rank $n-2$, denoted ${\rm AW}(n)$, is the subalgebra of $\Uqsl^{\otimes n}$ generated by all $\Lambda_A$ with $A\subseteq [1;n]$.
\end{Definition}

We will refer to the algorithm described in Definition \ref{right extension process def} as the right extension process, as to make the distinction with the following, alternative construction method.

\subsection{The left and mixed extension processes}

An alternative method to associate to each $A\subseteq[1;n]$ an element of $\Uqsl^{\otimes n}$ uses the left coideal comodule subalgebra $\I_L$ and its coaction $\tau_L$, as introduced in Definition \ref{I_L and tau_L def}.

\begin{Definition}
\label{left extension process def}
For any set $A = \{a_1,\dots,a_m\}\subseteq[1;n]$, ordered such that $a_i < a_{i+1}$ for all $i$, we define $\widehat{\Lambda}_A\in\Uqsl^{\otimes n}$ by
\[
\widehat{\Lambda}_A = 1^{\otimes (a_1-1)}\otimes \left(\overleftarrow{\prod_{i = 1}^{m-1}}\widehat{\mu}_i^{A} \right)(\Lambda)\otimes1^{\otimes (n-a_m)},
\]
with
\begin{gather}
\label{mu hat def}
\widehat{\mu}_i^{A} = \left(\overrightarrow{\prod_{\ell = a_m-a_{i+1}+1}^{a_m-a_i-1}}\big(\tau_L\otimes 1^{\otimes\ell}\big)\right)\big(\Delta\otimes 1^{\otimes (a_m-a_{i+1})}\big),
\end{gather}
where of course the term between brackets is absent if $a_{i+1} = a_i + 1$.
\end{Definition}
\begin{Example}
As before we take $n = 9$ and $A = \{2,4,5,8\}$, and find
\[
\widehat{\Lambda}_{\{2,4,5,8\}} = 1 \otimes \big(\widehat{\mu}_1^{\{2,4,5,8\}}\widehat{\mu}_2^{\{2,4,5,8\}}\widehat{\mu}_3^{\{2,4,5,8\}}\big)(\Lambda)\otimes 1,
\]
with $\widehat{\mu}_3^{\{2,4,5,8\}} = \big(\tau_L\otimes 1^{\otimes 2}\big)(\tau_L\otimes 1)\Delta$, $\widehat{\mu}_2^{\{2,4,5,8\}} = \Delta\otimes 1^{\otimes 3}$ and $\widehat{\mu}_1^{\{2,4,5,8\}} = \big(\tau_L\otimes 1^{\otimes 5}\big)\big(\Delta\otimes 1^{\otimes 4}\big)$.
\end{Example}

Again, the idea is that each element of $A$ but the maximum $a_m$ corresponds to an application of $\Delta$, whereas $\tau_L$ \emph{creates the holes}. However, as opposed to Definition \ref{right extension process def}, we now run through the elements of $A$ in decreasing order, from right to left. This is why we refer to the algorithm of Definition \ref{left extension process def} as the left extension process.

Our first major task will be to prove the equivalence of the right and left extension processes, i.e., to show that they produce the same elements, for each set $A$. To do so, it will often be needed to switch the order of certain algebra morphisms which act on mutually disjoint tensor product positions. More precisely, if $X\in\Uqsl^{\otimes 2}$ and $\varphi,\psi\colon \Uqsl\to\Uqsl^{\otimes 2}$, then we have the following basic property:
\begin{gather}\label{Lemma 1.2.1 eq}
(1\otimes1\otimes\psi)(\varphi\otimes1)X = (\varphi\otimes1\otimes 1)(1\otimes\psi)X.
\end{gather}

\begin{Remark}
\label{Remark comodules}
This property also allows to replace the definitions (\ref{mu def}) and (\ref{mu hat def}) of~$\mu_i^{A}$ and~$\widehat{\mu}_i^{A}$ by certain equivalent expressions. For example, $\mu_2^{\{1,5\}} = \big(1^{\otimes 3}\otimes\tau_R\big)\big(1^{\otimes 2}\otimes\tau_R\big)(1\otimes\tau_R)\Delta$. Invoking~(\ref{R comodule prop 1}) and (\ref{Lemma 1.2.1 eq}), we have
\begin{gather*}
\big(1^{\otimes2}\otimes\tau_R\big)(1\otimes\tau_R)\tau_R = \big(1^{\otimes 2}\otimes\tau_R\big)(\Delta\otimes 1)\tau_R \\
\hphantom{\big(1^{\otimes2}\otimes\tau_R\big)(1\otimes\tau_R)\tau_R}{} = \big(\Delta\otimes 1^{\otimes 2}\big)(1\otimes\tau_R)\tau_R = \big(\Delta\otimes 1^{\otimes 2}\big)(\Delta\otimes 1)\tau_R,
\end{gather*}
which allows to rewrite $\mu_2^{\{1,5\}}$, and of course this applies to any morphism of the form $\mu_i^{A}$ with $a_i-a_{i-1}>2$. Similarly, one finds from~(\ref{L comodule prop 1}) and~(\ref{Lemma 1.2.1 eq}) that
\[
\big(\tau_L\otimes 1^{\otimes 2}\big)(\tau_L\otimes 1)\tau_L = \big(1^{\otimes 2}\otimes \Delta\big)(1\otimes\Delta)\tau_L
\]
and its generalizations for any $\widehat{\mu}_{i}^{A}$ with $a_{i+1}-a_i>2$.
\end{Remark}

This observation will now help us show the equivalence of both extension processes.

\begin{Proposition}\label{thm - left and right extension process}The right and left extension processes produce exactly the same generators: for each $A\subseteq [1;n]$ one has $\Lambda_A = \widehat{\Lambda}_A$.
\end{Proposition}
\begin{proof}Without loss of generality, we can assume that $\min(A) = 1$ and $\max(A) = n$, since otherwise it suffices to add $1$ in the remaining positions. We write $A = [1;j_1]\cup[i_2;j_2]\cup\dots\cup[i_m;n]$ and proceed by induction on $m$. The case $m = 1$ follows from coassociativity:
\[
\Lambda_{[1;n]} = \big(1^{\otimes (n-2)}\otimes\Delta\big)\cdots(1\otimes\Delta)\Delta(\Lambda)= \big(\Delta\otimes 1^{\otimes (n-2)}\big)\cdots(\Delta\otimes 1)\Delta(\Lambda) = \widehat{\Lambda}_{[1;n]}.
\]
Suppose now the claim is true for all sets consisting of $m-1$ discrete intervals, including
$\widehat{A} = [1;j_1]\cup[i_2;j_2]\cup\dots\cup[i_{m-1};j_{m-1}+1]$: $\Lambda_{\widehat{A}} = \widehat{\Lambda}_{\widehat{A}}$. The right extension process asserts
\begin{gather}\label{left extension IH}
\Lambda_A = \big(1^{\otimes (n-2)}\otimes\Delta\big)\cdots\big(1^{\otimes (i_m-1)}\otimes\Delta\big)\big(1^{\otimes (i_{m}-2)}\otimes\tau_R\big)\cdots\big(1^{\otimes j_{m-1}}\otimes\tau_R\big)\Lambda_{\widehat{A}}.
\end{gather}
On the other hand, by the left extension process we have
\begin{gather}\label{Gamma of intermediary set}
\Lambda_{\widehat{A}} = \widehat{\Lambda}_{\widehat{A}} = \big(\alpha_{j_{m-1}-1}\otimes 1^{\otimes (j_{m-1}-1)}\big)\cdots(\alpha_1\otimes 1)\Delta(\Lambda),
\end{gather}
with each $\alpha_i \in \{\Delta, \tau_L\}$. When combining~(\ref{left extension IH}) and~(\ref{Gamma of intermediary set}), it is clear that the second tensor product position in $\Delta(\Lambda)$ is left invariant by all the $\alpha_i$. Hence by~(\ref{Lemma 1.2.1 eq}) we may shift the morphisms in~(\ref{left extension IH}) through those in (\ref{Gamma of intermediary set}), such that
\begin{gather}\label{Lambda A and B}
\Lambda_A = \big(\alpha_{j_{m-1}-1}\otimes 1^{\otimes (n-2)}\big)\cdots\big(\alpha_1\otimes 1^{\otimes (n-j_{m-1})}\big)\Lambda_B,
\end{gather}
with
\[
\Lambda_B = \big(1^{\otimes (n-j_{m-1}-1)}\otimes\Delta\big)\cdots\big(1^{\otimes (i_m-j_{m-1})}\otimes\Delta\big)\big(1^{\otimes (i_{m}-j_{m-1}-1)}\otimes\tau_R\big)\cdots(1\otimes\tau_R)\Delta(\Lambda).
\]
By (\ref{cotensor product property}) and (\ref{Lemma 1.2.1 eq}) we have
\[
(1\otimes 1\otimes\tau_R)(1\otimes\tau_R)\Delta(\Lambda) = (\tau_L\otimes 1\otimes 1)(1\otimes\tau_R)\Delta(\Lambda) = (\tau_L\otimes 1\otimes 1)(\tau_L\otimes 1)\Delta(\Lambda).
\]
Repeating this, we find
\[
\Lambda_B = \big(1^{\otimes (n-j_{m-1}-1)}\otimes\Delta\big)\cdots\big(1^{\otimes (i_m-j_{m-1})}\otimes\Delta\big)\big(\tau_L\otimes 1^{\otimes (i_{m}-j_{m-1}-1)}\big)\cdots(\tau_L\otimes 1)\Delta(\Lambda).
\]
Invoking (\ref{Lemma 1.2.1 eq}) again, we may shift all the $1^{\otimes\ell}\otimes\Delta$ through all the $\tau_L\otimes 1^{\otimes m}$, which, after using coassociativity, gives
\begin{gather*}
\Lambda_B = \big(\tau_L\otimes 1^{\otimes (n-j_{m-1}-1)}\big)\cdots\big(\tau_L\otimes 1^{\otimes(n-i_m+1)}\big)\big(\Delta\otimes 1^{\otimes (n-i_m)}\big)\cdots(\Delta\otimes 1)\Delta(\Lambda) \\
\hphantom{\Lambda_B}{}= \widehat{\Lambda}_{\{1\}\cup[i_m-j_{m-1}+1;n-j_{m-1}+1]}.
\end{gather*}
Moreover, the left extension process and (\ref{Gamma of intermediary set}) imply that
\[
\widehat{\Lambda}_A = \big(\alpha_{j_{m-1}-1}\otimes 1^{\otimes (n-2)}\big)\cdots\big(\alpha_1\otimes 1^{\otimes (n-j_{m-1})}\big)\widehat{\Lambda}_{\{1\}\cup[i_m-j_{m-1}+1;n-j_{m-1}+1]}.
\]
The statement now follows from (\ref{Lambda A and B}).
\end{proof}

The reasoning established in the proof of Proposition~\ref{thm - left and right extension process} suggests the existence of several other, so-called mixed extension processes, which produce the same elements $\Lambda_A$. To be precise, one can split the set $A$ at any $a_j$ and perform the right extension process for the subset $\{a_{j+1},\dots,a_m\}$ and consequently the left extension process for $\{a_1,\dots,a_{j-1}\}$. This is described in the following definition.

\begin{Definition}
\label{mixed extension process def}
For any set $A= \{a_1,\dots,a_m\}\subseteq [1;n]$, ordered such that $a_i < a_{i+1}$ for all $i$, we define $\Lambda_A^{(j)}\in\Uqsl^{\otimes n}$ by
\[
\Lambda_A^{(j)} = 1^{\otimes (a_1-1)}\otimes\left(\overleftarrow{\prod_{i = 1}^{j-1}}\widehat{\mu}_i^{A}\right)\left(\overrightarrow{\prod_{i = j+1}^{m}\mu_{i,j}^{A}}\right)(\Lambda)\otimes 1^{\otimes (n-a_m)},
\]
with $\widehat{\mu}_i^{A}$ as in (\ref{mu hat def}) and
\[
\mu_{i,j}^{A} = \left(\overrightarrow{\prod_{\ell = a_{i-1}-a_j+1}^{a_i-a_j-1}}\big(1^{\otimes\ell}\otimes\tau_R\big)\right)\big(1^{\otimes (a_{i-1}-a_j)}\otimes\Delta\big).
\]
\end{Definition}

By definition, we have $\Lambda_A^{(1)}=\Lambda_A$ and $\Lambda_A^{(m)} = \widehat{\Lambda}_A$ and following the proof of Proposition~\ref{thm - left and right extension process}, we even have $\Lambda_A^{(j)} = \Lambda_A$ for all $j\in\{1,\dots,m\}$.

\begin{Example}
For $n = 9$, $A = \{2,4,5,8\}$ and $j = 2$ we have
\[
\Lambda_{\{2,4,5,8\}}^{(2)} = 1\otimes \big(\widehat{\mu}_1^{\{2,4,5,8\}}\mu_{4,2}^{\{2,4,5,8\}}\mu_{3,2}^{\{2,4,5,8\}}\big)(\Lambda)\otimes 1,
\]
with $\mu_{3,2}^{\{2,4,5,8\}} = \Delta$, $\mu_{4,2}^{\{2,4,5,8\}} = \big(1^{\otimes 3}\otimes\tau_R\big)\big(1^{\otimes 2}\otimes\tau_R\big)(1\otimes\Delta) $ and $\widehat{\mu}_1^{\{2,4,5,8\}} = \big(\tau_L\otimes 1^{\otimes 5}\big)\big(\Delta\otimes 1^{\otimes 4}\big)$.
\end{Example}

\subsection{Another layer of freedom}

In the extension processes described above, the order in which we apply the different morphisms is of high importance. Nevertheless, there is some additional freedom in this order, which will come in handy in many of the following proofs. More precisely, the upcoming Proposition~\ref{prop derived} asserts that when constructing $\Lambda_A$, it suffices to first create all the holes between elements of $A$ in order of appearance and then enlarge all holes and all intervals by applying repeatedly the coproduct $\Delta$.

\begin{Proposition}\label{prop derived}Let $A = [1;j_1]\cup[i_2;j_2]\cup\dots\cup[i_k;j_k]$, then one has
\begin{gather}\label{prop derived TBP}
\Lambda_A = \left(\overleftarrow{\prod_{n=0}^{2k-2}}\overrightarrow{\prod_{\ell=\alpha_{n,\mathbf{i},\mathbf{j}}}^{\beta_{n,\mathbf{i},\mathbf{j}}}}(1^{\otimes n}\otimes\Delta\otimes 1^{\otimes \ell})\right)\Lambda_{\{1,3,5,\dots,2k-1\}},
\end{gather}
where \[
\alpha_{n,\mathbf{i},\mathbf{j}} =
\begin{cases}
j_k-j_m & \text{if}\ n = 2m-2\ \text{is even}, \\
j_k-i_m+1 & \text{if}\ n = 2m-3\ \text{is odd}
\end{cases}
\]
and
\[
\beta_{n,\mathbf{i},\mathbf{j}} =
\begin{cases}
j_k-i_m-1 & \text{if}\ n = 2m-2\ \text{is even}, \\
j_k-j_{m-1}-2 & \text{if}\ n = 2m-3 \ \text{is odd}
\end{cases}
\]
and where we set $\mathbf{i} = (1,i_2,\dots,i_k)$ and $\mathbf{j} = (j_1,j_2,\dots,j_k)$.
\end{Proposition}
\begin{proof}By induction on $k$, the case $k = 1$ being trivial by coassociativity. Suppose hence the claim holds for all sets consisting of $k-1$ discrete intervals, including $\widehat{A} =[1;j_1]\cup[i_2;j_2]\cup\dots\cup[i_{k-1};j_{k-1}]$:
\begin{gather}\label{derived IH}
\Lambda_{\widehat{A}} = \left(\overleftarrow{\prod_{n=0}^{2k-4}}\overrightarrow{\prod_{\ell= \alpha_{n,\widetilde{\mathbf{i}},\widetilde{\mathbf{j}}}}^{\beta_{n,\widetilde{\mathbf{i}},\widetilde{\mathbf{j}}}}}\big(1^{\otimes n}\otimes\Delta\otimes 1^{\otimes \ell}\big)\right)\Lambda_{\{1,3,5,\dots,2k-3\}},
\end{gather}
where $\widetilde{\mathbf{i}} = (1,i_2,\dots,i_{k-1})$ and $\widetilde{\mathbf{j}} = (j_1,j_2,\dots,j_{k-1})$.
From the right extension process, using coassociativity and Remark~\ref{Remark comodules}, it is clear that
\begin{gather}
\Lambda_A = \big(1^{\otimes (i_{k}-1)}\otimes\Delta\otimes 1^{\otimes (j_k-i_k-1)}\big)\cdots\big(1^{\otimes (i_{k}-1)}\otimes\Delta\big)\big(1^{\otimes j_{k-1}}\otimes\Delta\otimes 1^{\otimes (i_k-j_{k-1}-2)}\big)\cdots\nonumber\\
\hphantom{\Lambda_A =}{}\times \big(1^{\otimes j_{k-1}}\otimes\Delta\otimes 1\big)\big(1^{\otimes j_{k-1}}\otimes\tau_R\big)\big(1^{\otimes(j_{k-1}-1)}\otimes\Delta\big)\Lambda_{\widehat{A}}.\label{derived right extension procedure}
\end{gather}
Observe that when combining (\ref{derived IH}) and (\ref{derived right extension procedure}), the morphisms in (\ref{derived IH}) with $n \leq 2k-5$ will act on tensor product positions in $\Lambda_{\{1,3,5,\dots,2k-3\}}$ of lower index than the morphisms in the second line in (\ref{derived right extension procedure}). We may hence switch their order as in (\ref{Lemma 1.2.1 eq}):
\begin{gather}
\label{partial result after switching order}
\big(1^{\otimes j_{k-1}}\otimes\tau_R\big)\big(1^{\otimes(j_{k-1}-1)}\otimes\Delta\big)\Lambda_{\widehat{A}} = \left(\overleftarrow{\prod_{n=0}^{2k-5}}\overrightarrow{\prod_{\ell= \alpha_{n,\widetilde{\mathbf{i}},\widetilde{\mathbf{j}}}+2}^{\beta_{n,\widetilde{\mathbf{i}},\widetilde{\mathbf{j}}}+2}}\big(1^{\otimes n}\otimes\Delta\otimes 1^{\otimes \ell}\big)\right)\Lambda_B,
\end{gather}
where $\Lambda_B$ is defined as $\chi(\Lambda_{\{1,3,5,\dots,2k-3\}})$, where $\chi$ is the morphism
\begin{gather*}
\big(1^{\otimes (j_{k-1}-i_{k-1}+2k-3)}\otimes\tau_R\big)\big(1^{\otimes(j_{k-1}-i_{k-1}+2k-4)}\otimes\Delta\big)\left(\overrightarrow{\prod_{\ell= 0}^{j_{k-1}-i_{k-1}-1}}\big(1^{\otimes (2k-4)}\otimes\Delta\otimes 1^{\otimes \ell}\big)\right).
\end{gather*}
Here, the term between brackets corresponds to $n = 2k-4$ in (\ref{derived IH}), since
\begin{gather}\label{alpha and beta for 2k-4}
\alpha_{2k-4,\widetilde{\mathbf{i}},\widetilde{\mathbf{j}}} = 0 \qquad\mathrm{and}\qquad \beta_{2k-4,\widetilde{\mathbf{i}},\widetilde{\mathbf{j}}} = j_{k-1}-i_{k-1}-1.
\end{gather}
Using coassociativity, separating the term for $\ell = 0$ and relying on the right extension process, we have
\begin{gather*}
\Lambda_B
= \big(1^{\otimes (j_{k-1}-i_{k-1}+2k-3)}\otimes\tau_R\big)\left(\overrightarrow{\prod_{\ell= 1}^{j_{k-1}-i_{k-1}}}\big(1^{\otimes (2k-4)}\otimes\Delta\otimes 1^{\otimes \ell}\big)\right)\Lambda_{\{1,3,5,\dots,2k-3,2k-2\}}.
\end{gather*}
Now it is manifest that all $\Delta$ in the product between brackets act on the tensor product position $2k-3$ in $\Lambda_{\{1,3,5,\dots,2k-3,2k-2\}}$, whereas $\tau_R$ in fact acts on the last position $2k-2$. Hence we may again apply~(\ref{Lemma 1.2.1 eq}) to switch the order:
\[
\Lambda_B = \left(\overrightarrow{\prod_{\ell= 2}^{j_{k-1}-i_{k-1}+1}}\big(1^{\otimes (2k-4)}\otimes\Delta\otimes 1^{\otimes \ell}\big)\right)\Lambda_{\{1,3,5,\dots,2k-3,2k-1\}}.
\]
Note that the lower and upper bounds in the product between brackets equal $\alpha_{2k-4,\widetilde{\mathbf{i}},\widetilde{\mathbf{j}}}+2$ and $\beta_{2k-4,\widetilde{\mathbf{i}},\widetilde{\mathbf{j}}}+2$ respectively, by~(\ref{alpha and beta for 2k-4}).

Combined with (\ref{derived right extension procedure}) and (\ref{partial result after switching order}), this yields
\begin{gather*}
\Lambda_A = \big(1^{\otimes (i_{k}-1)}\otimes\Delta\otimes 1^{\otimes (j_k-i_k-1)}\big)\cdots\big(1^{\otimes (i_{k}-1)}\otimes\Delta\big)\big(1^{\otimes j_{k-1}}\otimes\Delta\otimes 1^{\otimes (i_k-j_{k-1}-2)}\big)\cdots\\
\hphantom{\Lambda_A =}{}\times \big(1^{\otimes j_{k-1}}\otimes\Delta\otimes 1\big)\left(\overleftarrow{\prod_{n=0}^{2k-4}}\overrightarrow{\prod_{\ell= \alpha_{n,\widetilde{\mathbf{i}},\widetilde{\mathbf{j}}}+2}^{\beta_{n,\widetilde{\mathbf{i}},\widetilde{\mathbf{j}}}+2}}\big(1^{\otimes n}\otimes\Delta\otimes 1^{\otimes \ell}\big)\right)\Lambda_{\{1,3,5,\dots,2k-3,2k-1\}}.
\end{gather*}
Moreover, all morphisms in the first line act on the tensor product positions $2k-2$ or $2k-1$ of $\Lambda_{\{1,3,5,\dots,2k-3,2k-1\}}$, whereas those on the second line act on positions 1 to $2k-3$. The order can hence be switched again by~(\ref{Lemma 1.2.1 eq}), leading first to
\begin{gather*}
\Lambda_A = \left(\overleftarrow{\prod_{n=0}^{2k-4}}\overrightarrow{\prod_{\ell= \alpha_{n,\widetilde{\mathbf{i}},\widetilde{\mathbf{j}}}+j_k-j_{k-1}}^{\beta_{n,\widetilde{\mathbf{i}},\widetilde{\mathbf{j}}}+j_k-j_{k-1}}}\big(1^{\otimes n}\otimes\Delta\otimes 1^{\otimes \ell}\big)\right)\\
\hphantom{\Lambda_A =}{}\times \big(1^{\otimes (i_{k}-j_{k-1}+2k-4)}\otimes\Delta\otimes 1^{\otimes (j_k-i_k-1)}\big)\cdots\big(1^{\otimes (i_{k}-j_{k-1}+2k-4)}\otimes\Delta\big)\\
\hphantom{\Lambda_A =}{}\times
\big(1^{\otimes (2k-3)}\otimes\Delta\otimes 1^{\otimes (i_k-j_{k-1}-2)}\big)\cdots\big(1^{\otimes (2k-3)}\otimes\Delta\otimes 1\big)\Lambda_{\{1,3,5,\dots,2k-3,2k-1\}}
\end{gather*}
and then, after switching the morphisms on the second with those on the third line, to
\begin{gather*}
\Lambda_A = \left(\overleftarrow{\prod_{n=0}^{2k-4}}\overrightarrow{\prod_{\ell= \alpha_{n,\widetilde{\mathbf{i}},\widetilde{\mathbf{j}}}+j_k-j_{k-1}}^{\beta_{n,\widetilde{\mathbf{i}},\widetilde{\mathbf{j}}}+j_k-j_{k-1}}}\big(1^{\otimes n}\otimes\Delta\otimes 1^{\otimes \ell}\big)\right)\\
\hphantom{\Lambda_A =}{}\times \big(1^{\otimes (2k-3)}\otimes\Delta\otimes 1^{\otimes (j_k-j_{k-1}-2)}\big)\cdots\big(1^{\otimes (2k-3)}\otimes\Delta\otimes 1^{\otimes (j_k-i_k+1)}\big)\\
\hphantom{\Lambda_A =}{}\times
\big(1^{\otimes (2k-2)}\otimes\Delta\otimes 1^{\otimes (j_k-i_k-1)}\big)\cdots\big(1^{\otimes (2k-2)}\otimes\Delta\big)
\Lambda_{\{1,3,5,\dots,2k-3,2k-1\}}.
\end{gather*}
Now observe that for every $n\in \{0,1,\dots,2k-4\}$ one has $
\alpha_{n,\mathbf{i},\mathbf{j}} = \alpha_{n,\widetilde{\mathbf{i}},\widetilde{\mathbf{j}}}+j_k-j_{k-1}$ and $\beta_{n,\mathbf{i},\mathbf{j}} = \beta_{n,\widetilde{\mathbf{i}},\widetilde{\mathbf{j}}}+j_k-j_{k-1}$, and moreover we have $\alpha_{2k-3,\mathbf{i},\mathbf{j}} = j_k-i_k+1$, $\beta_{2k-3,\mathbf{i},\mathbf{j}} = j_k-j_{k-1}-2$, $\alpha_{2k-2,\mathbf{i},\mathbf{j}}= 0$ and $\beta_{2k-2,\mathbf{i},\mathbf{j}} = j_k-i_k-1$.
Hence the expression above coincides with (\ref{prop derived TBP}).
\end{proof}

\section{Main results and strategy of proof}\label{Section - Main results}

In this section, we formulate the main results of this paper: the algebraic relations satisfied in the higher rank Askey--Wilson algebra ${\rm AW}(n)$. As in the rank one case, these will be of the form
\begin{gather}\label{standard relation}
[\Lambda_A,\Lambda_B]_q = \big(q^{-2}-q^2\big)\Lambda_{(A\cup B)\setminus(A\cap B)} + \big(q-q^{-1}\big)\big(\Lambda_{A\cap B}\Lambda_{A\cup B} + \Lambda_{A\setminus(A\cap B)}\Lambda_{B\setminus(A\cap B)}\big), \tag{$\ast$}
\end{gather}
or
\begin{gather}\label{commutation}
[\Lambda_A,\Lambda_B] = 0, \tag{$\Delta$}
\end{gather}
under suitable conditions on the sets $A$ and $B$. In this section, we present these conditions, which, based on extensive computer calculations, we believe to be minimal. They can be stated as follows.

\begin{Theorem}\label{thm - commutation general}Let $A,B\subseteq[1;n]$ be such that $B\subseteq A$, then $\Lambda_A$ and $\Lambda_B$ commute.
\end{Theorem}
Section \ref{section - proof commutation general} will be devoted to the proof of this theorem.

\begin{Definition}
\label{Notation prec}
For $A,B\subseteq [1;n]$ we write $A\prec B$ if $\max(A)<\min(B)$ or if either $A$ or $B$ is empty.
\end{Definition}

\begin{Theorem}
\label{thm - most general algebra relations}
Let $A_1$, $A_2$, $A_3$ and $A_4$ be $($potentially empty$)$ subsets of $[1;n]$ satisfying
\[
A_1 \prec A_2 \prec A_3 \prec A_4.
\]
The standard relation
\[
[\Lambda_A,\Lambda_B]_q = \big(q^{-2}-q^2\big)\Lambda_{(A\cup B)\setminus(A\cap B)} + \big(q-q^{-1}\big)\big(\Lambda_{A\cap B}\Lambda_{A\cup B} + \Lambda_{A\setminus(A\cap B)}\Lambda_{B\setminus(A\cap B)}\big)
\]
is satisfied for $A$ and $B$ defined by one of the following relations:
\begin{alignat}{3}\label{most general sets 1}
& A= A_1 \cup A_2 \cup A_4, \qquad && B = A_2 \cup A_3,& \\
\label{most general sets 2}
& A= A_2 \cup A_3, \qquad && B = A_1 \cup A_3 \cup A_4, &\\
\label{most general sets 3}
& A= A_1 \cup A_3 \cup A_4, \qquad && B = A_1 \cup A_2 \cup A_4.&
\end{alignat}
\end{Theorem}
This will be shown in Section \ref{section - proof most general algebra relations}.

Our general strategy to prove a relation of the form (\ref{standard relation}) will be as follows. First we will construct an operator $\chi$, by combining morphisms of the form $1^{\otimes n}\otimes\alpha\otimes 1^{\otimes m}$, $\alpha\in\{\Delta,\tau_R,\tau_L\}$ and $n,m\in\N$, such that
\[
\chi(\Lambda_{A'}) = \Lambda_A, \qquad \chi(\Lambda_{B'}) = \Lambda_B, \qquad \chi(\Lambda_{(A'\cup B')\setminus(A'\cap B')}) = \Lambda_{(A\cup B)\setminus(A\cap B)}, \qquad\dots
\]
for certain (less complicated) sets $A'$, $B'$. To prove (\ref{standard relation}) it now suffices to show
\begin{gather}
[\Lambda_{A'},\Lambda_{B'}]_q = \big(q^{-2}-q^2\big)\Lambda_{(A'\cup B')\setminus(A'\cap B')}
\nonumber\\
\hphantom{[\Lambda_{A'},\Lambda_{B'}]_q =}{}
+ \big(q-q^{-1}\big)\big(\Lambda_{A'\cap B'}\Lambda_{A'\cup B'} + \Lambda_{A'\setminus(A'\cap B')}\Lambda_{B'\setminus(A'\cap B')}\big), \tag{$\ast\ast$}\label{mock A}
\end{gather}
apply the operator $\chi$ to both sides of the equation and use its linearity and multiplicativity. In this case we will write ``(\ref{standard relation}) follows from (\ref{mock A}) by $\chi$''. The same strategy applies to relations of the form~(\ref{commutation}). In this respect, the relations of Theorem~\ref{thm - most general algebra relations} can in fact be derived from the following 9 fundamental cases.

\begin{Proposition}\label{prop - fundamental relations}The standard relation \eqref{standard relation} holds for the following combinations of sets~$A$ and~$B$:{\samepage
\begin{gather}\label{sets 1}
\tag{C1}
\begin{split}
&A = \{1,2,4,6,\dots,2k\}, \\ & B = \{2,4,6,\dots,2k,2k+1\};
\end{split}\vspace{2mm}\\
\label{sets 2}
\tag{C2}
\begin{split}
& A= \{2,4,6,\dots,2k,2k+1\}, \\ & B= \{1,2k+1\};
\end{split}\vspace{2mm}\\
\label{sets 3}
\tag{C3}
\begin{split}
& A= \{1,2k+1\}, \\ & B= \{1,2,4,6,\dots,2k\};
\end{split}\vspace{2mm}\\
\label{sets 4 without gap}
\tag{C4}
\begin{split}
& A= \{1,2,4,6,\dots,2k,2k+2\ell+2\}, \\ & B= \{2,4,6,\dots,2k,2k+1,2k+3,\dots,2k+2\ell+1\};
\end{split}\vspace{2mm}\\
\label{sets 4 with gap}\tag{C4$'$}
\begin{split}
& A= \{1,2,4,6,\dots,2k,2k+2\ell+3\}, \\ & B= \{2,4,6,\dots,2k,2k+2,2k+4,\dots,2k+2\ell+2\};
\end{split}\vspace{2mm}\\
\label{sets 5 without gap}
\tag{C5}
\begin{split}
& A= \{2,4,6,\dots,2k,2k+1,2k+3,\dots,2k+2\ell+1\}, \\ & B= \{1,2k+1,2k+3,\dots,2k+2\ell+1,2k+2\ell+2\};
\end{split}\vspace{2mm}\\
\label{sets 5 with gap}\tag{C5$'$}
\begin{split}
& A= \{2,4,6,\dots,2k,2k+2,2k+4,\dots,2k+2\ell+2\}, \\ & B= \{1,2k+2,2k+4,\dots,2k+2\ell+2,2k+2\ell+3\};
\end{split}\vspace{2mm}\\
\label{sets 6 without gap}
\tag{C6}
\begin{split}
& A= \{1,2k+1,2k+3,\dots,2k+2\ell+1,2k+2\ell+2\}, \\ & B= \{1,2,4,6,\dots,2k,2k+2\ell+2\};
\end{split}\vspace{2mm}\\
\label{sets 6 with gap}\tag{C6$'$}
\begin{split}
& A= \{1,2k+2,2k+4,\dots,2k+2\ell+2,2k+2\ell+3\}, \\ & B= \{1,2,4,6,\dots,2k,2k+2\ell+3\},
\end{split}
\end{gather}
with $k,\ell\in\N$.}
\end{Proposition}

In the next section, we will provide an elaborate proof for each of these fundamental relations. They will serve as the building blocks to prove similar fundamental commutation relations in Section~\ref{section - more commutation relations} and consequently to prove Theorems~\ref{thm - commutation general} and~\ref{thm - most general algebra relations} in Sections~\ref{section - proof commutation general} and~\ref{section - proof most general algebra relations} respectively.

\section{Proof of Proposition \ref{prop - fundamental relations}}\label{section - proof fundamental relations}

\subsection{Some basic commutation relations}

Throughout the whole paper, it will turn out useful to switch orders in nested $q$-commutators. A straightforward calculation gives the following:

\begin{Lemma}Let $\mathcal{A}$ be any algebra and $\alpha$, $\beta$, $\gamma$, $\delta$ elements of $\mathcal{A}$, then one has
\begin{gather}\label{q-anticomm 1}
[\alpha,[\gamma,\delta]_q]_q = [[\alpha,\gamma]_q,\delta]_q \qquad \mathrm{if} \quad [\alpha,\delta] = 0, \\
\label{q-anticomm 2}
[\alpha,[\gamma,\delta]_q]_q = [\gamma,[\alpha,\delta]_q]_q \qquad \mathrm{if} \quad [\alpha,\gamma] = 0, \\
\label{q-anticomm 3}
[[\gamma,\delta]_q,\beta]_q = [[\gamma,\beta]_q,\delta]_q \qquad \mathrm{if} \quad [\beta,\delta] = 0, \\
\label{q-anticomm 4}
[[\gamma,\delta]_q,\beta]_q = [\gamma,[\delta,\beta]_q]_q \qquad \mathrm{if} \quad [\beta,\gamma] = 0.
\end{gather}
\end{Lemma}

The following commutation relations are so natural that they will often be relied on in proofs in Subsection \ref{subsection - fundamental cases} without explicit reference.

\begin{Lemma}\label{lemma Gamma of singleton}For any $i\in[1;n]$ and any subset $A\subseteq [1;n]$ one has
\[
[\Lambda_A, \Lambda_{\{i\}}] = 0.
\]
\end{Lemma}
\begin{proof}This is immediate by the fact that $\Lambda_{\{i\}} = 1^{\otimes (i-1)}\otimes\Lambda\otimes 1^{\otimes(n-i)}$ and that $\Lambda$ is the Casimir operator of $U_q(\mathfrak{sl}_2)$.
\end{proof}

\begin{Lemma}\label{Lemma 1.3.7}Let $A, A'\subseteq [1;n]$ be such that $A \prec A'$, where we use Definition~{\rm \ref{Notation prec}}, then one has
\[
[\Lambda_A,\Lambda_{A'}] = 0, \qquad [\Lambda_A,\Lambda_{A\cup A'}] = 0,\qquad [\Lambda_{A'},\Lambda_{A\cup A'}] = 0.
\]
\end{Lemma}
\begin{proof}The first statement is trivial, since $\Lambda_A$ and $\Lambda_{A'}$ live in disjoint tensor product positions. For the second claim, writing
\[
A = [1;j_1]\cup[i_2;j_2]\cup\cdots\cup[i_k;j_k], \qquad A' = [i_{k+1};j_{k+1}]\cup\cdots\cup[i_{k+\ell};j_{k+\ell}],
\]
the statement follows from $[\Lambda_{\{1\}},\Lambda_{\{1,2\}}] = 0$ by a suitable morphism of the form
\begin{gather*}
\chi = \left(\overrightarrow{\prod_{m=i_{k+1}-1}^{j_{k+\ell}-2}}\big(1^{\otimes m}\otimes\beta_{m}\big)\right)\left(\overrightarrow{\prod_{m=1}^{i_{k+1}-2}}\big(\alpha_m\otimes 1^{\otimes m}\big)\right),
\end{gather*}
where each $\alpha_m\in\{\Delta,\tau_L\}$ and each $\beta_{m} \in \{\Delta, \tau_R\}$. The third statement follows analogously.
\end{proof}

The next lemma provides a first generalization of the relations (\ref{rank one 12-23})--(\ref{rank one 13-12}).

\begin{Lemma}\label{Corollary 1.3.4}Let $i\in[1;n]$ and $A_1,A_2\subseteq[1;n]$ be such that $A_1\prec\{i\}\prec A_2$, then the relation~\eqref{standard relation} holds for $(A,B)$ one of the couples
\[
(A_1 \cup \{i\}, \{i\}\cup A_2), \qquad (\{i\}\cup A_2, A_1\cup A_2), \qquad (A_1\cup A_2, A_1\cup \{i\}).
\]
\end{Lemma}
\begin{proof}Let $A_1 \cup \{i\}\cup A_2 = \{a_1,\dots,a_m\}$, ordered such that $a_{\ell}<a_{\ell+1}$ for all $\ell$ and let $j$ be such that $a_j = i$. The mixed extension process with parameter $j$ asserts the existence of a~morphism~$\chi$ which sends
\begin{gather*}
\{1\}\mapsto A_1, \qquad \{2\}\mapsto \{i\}, \qquad \{3\}\mapsto A_2.
\end{gather*}
Hence the statements follow from (\ref{rank one 12-23})--(\ref{rank one 13-12}) by $\chi$.
\end{proof}

A final immediate commutation relation is the following.

\begin{Lemma}\label{Lemma 1.3.8}For any $k\in\N$ one has
\[
[\Lambda_{\{2,4,\dots,2k\}},\Lambda_{\{1,2k+1\}}] = 0.
\]
\end{Lemma}
\begin{proof} By induction on $k$. The case $k = 1$ is trivial by Lemma \ref{lemma Gamma of singleton}. Suppose hence the claim has been proven for $k-1$. Lemma~\ref{Corollary 1.3.4} asserts
\[
\Lambda_{\{2,4,6,\dots,2k\}} = \frac{[\Lambda_{\{2,3\}},\Lambda_{\{3,4,6,\dots,2k \}}]_q}{q^{-2}-q^2}+\frac{\Lambda_{\{3\}}\Lambda_{\{2,3,4,6,\dots,2k\}} + \Lambda_{\{2\}}\Lambda_{\{4,6,\dots,2k\}}}{q+q^{-1}},
\]
hence it suffices to show that $\Lambda_{\{1,2k+1\}} $ commutes with each term in the right-hand side. For~$\Lambda_{\{2,3\}}$ this follows from $[\Lambda_{\{1,3\}},\Lambda_{\{2\}}] = 0$ by $\chi = \big(1^{\otimes (2k-1)}\otimes\tau_R)\cdots\big(1^{\otimes 3}\otimes\tau_R\big)(1\otimes\Delta\otimes 1)$. The other nontrivial commutation relations follow from the induction hypothesis, by $\chi = \big(\tau_L\otimes 1^{\otimes (2k-1)}\big)\big(1\otimes\Delta\otimes 1^{\otimes (2k-3)}\big)$, $\chi = \big(1\otimes\Delta\otimes 1^{\otimes (2k-2)}\big)\big(1\otimes\Delta\otimes 1^{\otimes (2k-3)}\big)$ and $\chi = \big(\tau_L\otimes 1^{\otimes (2k-3)}\big)\big(\tau_L\otimes 1^{\otimes (2k-2)}\big)$ respectively.
\end{proof}

\subsection{The fundamental cases (\ref{sets 1})--(\ref{sets 6 with gap})}\label{subsection - fundamental cases}

In this section we will prove that the relation (\ref{standard relation}) is satisfied for the combinations of sets (\ref{sets 1})--(\ref{sets 6 with gap}). We will work out the proof for three of these cases in detail, and describe concisely how one can show the remaining cases.

\begin{Lemma}The relation~\eqref{standard relation} holds for the sets \eqref{sets 2} and \eqref{sets 5 without gap} with $\ell = 0$.
\end{Lemma}
\begin{proof}We will prove both claims together in one single induction on $k$. For $k = 1$ the first claim coincides with (\ref{rank one 23-13}) and the second follows by direct calculation.

Our induction hypothesis states that (\ref{standard relation}) holds for
\begin{gather}\label{IH 1}
A = \{2,4,\dots,2k-2,2k-1\}, \qquad B = \{1,2k-1\}
\end{gather}
and
\begin{gather}\label{IH 2}
A = \{2,4,\dots,2k-2,2k-1\}, \qquad B = \{1,2k-1,2k\}.
\end{gather}
By $\chi = \big(1^{\otimes (2k-3)}\otimes\Delta\otimes 1\otimes 1\big)\big(1^{\otimes (2k-3)}\otimes\Delta\otimes 1\big)$, (\ref{IH 1}) implies (\ref{standard relation}) for
\begin{gather}\label{IH 3}
A = \{2,4,\dots,2k-2,2k-1,2k,2k+1\}, \qquad B = \{1,2k+1\}
\end{gather}
and by $\chi = 1^{\otimes (2k-1)}\otimes \Delta$, (\ref{IH 2}) gives rise to (\ref{standard relation}) for
\begin{gather}\label{IH 4}
A = \{2,4,\dots,2k-2,2k-1\}, \qquad B = \{1,2k-1,2k,2k+1\}.
\end{gather}

We will first compute
\begin{gather}\label{to compute 1}
[\Lambda_{\{2,4,6,\dots,2k,2k+1\}},\Lambda_{\{1,2k+1\}}]_q.
\end{gather}

By Lemma \ref{Lemma 1.3.8} we have (\ref{commutation}) for $A = \{1,2k-1\}$, $B = \{2,4,\dots,2k-2\}$, which, upon applying $\chi = \big(1^{\otimes (2k-1)}\otimes\tau_R\big)\big(1^{\otimes(2k-3)}\otimes\Delta\otimes 1\big)$, yields (\ref{commutation}) for
\begin{gather}\label{comm rel xx}
A = \{1,2k+1\}, \qquad B = \{2,4,\dots,2k-2,2k-1\},
\end{gather}
whereas by $\chi = \big(1^{\otimes (2k-1)}\otimes\tau_R\big)\big(1^{\otimes(2k-2)}\otimes\tau_R\big)$ we have (\ref{commutation}) for
\begin{gather}\label{comm rel x}
A = \{1,2k+1\}, \qquad B = \{2,4,\dots,2k-2\}.
\end{gather}
By Lemma \ref{Corollary 1.3.4} we may write
\begin{gather*}
\Lambda_{\{2,4,\dots,2k,2k+1\}} = \frac{[\Lambda_{\{2,4,\dots,2k-2,2k-1\}},\Lambda_{\{2k-1,2k,2k+1\}}]_q}{q^{-2}-q^2} \\
\hphantom{\Lambda_{\{2,4,\dots,2k,2k+1\}} =}{} +\frac{\Lambda_{\{2k-1\}}\Lambda_{\{2,4,\dots,2k-2,2k-1,2k,2k+1\}} + \Lambda_{\{2,4,\dots,2k-2\}}\Lambda_{\{2k,2k+1\}}}{q+q^{-1}}.
\end{gather*}
Substituting this in (\ref{to compute 1}) and using (\ref{q-anticomm 4}) by (\ref{comm rel xx}), (\ref{to compute 1}) becomes
\begin{gather*} \frac{[\Lambda_{\{2,4,\dots,2k-2,2k-1\}},[\Lambda_{\{2k-1,2k,2k+1\}},\Lambda_{\{1,2k+1\}}]_q]_q}{q^{-2}-q^2} \\
\quad {}+ \frac{\Lambda_{\{2k-1\}}[\Lambda_{\{2,4,\dots,2k-2,2k-1,2k,2k+1\}},\Lambda_{\{1,2k+1\}}]_q + \Lambda_{\{2,4,\dots,2k-2\}}[\Lambda_{\{2k,2k+1\}},\Lambda_{\{1,2k+1\}}]_q}{q+q^{-1}},
\end{gather*}
where we have used (\ref{comm rel x}) and Lemma \ref{Lemma 1.3.7}.

The relation (\ref{standard relation}) holds for $A = \{2k-1,2k,2k+1\}$, $B = \{1,2k+1\}$ and for $A = \{2k,2k+1\}$, $B = \{1,2k+1\}$, as one sees by applying $\chi = \big(\tau_L\otimes 1^{\otimes (2k-1)}\big)\cdots\big(\tau_L\otimes 1^{\otimes 3}\big)(1\otimes\Delta\otimes 1) $ resp.\ $\chi = \big(\tau_L\otimes 1^{\otimes (2k-1)}\big)\cdots\big(\tau_L\otimes 1^{\otimes 2}\big) $ to~(\ref{rank one 23-13}). With this and (\ref{IH 3}), (\ref{to compute 1}) becomes
\begin{gather*}
 [\Lambda_{\{2,4,\dots,2k-2,2k-1\}}, \Lambda_{\{1,2k-1,2k\}}]_q \\
{}- \frac{\Lambda_{\{2k+1\}}[\Lambda_{\{2,4,\dots,2k-2,2k-1\}},\Lambda_{\{1,2k-1,2k,2k+1\}}]_q+\Lambda_{\{1\}}[\Lambda_{\{2,4,\dots,2k-2,2k-1\}},\Lambda_{\{2k-1,2k\}}]_q}{q+q^{-1}}
\\ {}- \big(q-q^{-1}\big)\Lambda_{\{2k-1\}}\Lambda_{\{1,2,4,\dots,2k-2,2k-1,2k\}}
\\ {} +\frac{q-q^{-1}}{q+q^{-1}}\Lambda_{\{2k-1\}}\big(\Lambda_{\{2k+1\}}\Lambda_{\{1,2,4,\dots,2k-2,2k-1,2k,2k+1\}}+\Lambda_{\{1\}}\Lambda_{\{2,4,\dots,2k-2,2k-1,2k\}}\big)
\\ {} - \big(q-q^{-1}\big) \Lambda_{\{2,4,\dots,2k-2\}}\Lambda_{\{1,2k\}} + \frac{q-q^{-1}}{q+q^{-1}}\Lambda_{\{2,4,\dots,2k-2\}}\big(\Lambda_{\{2k+1\}}\Lambda_{\{1,2k,2k+1\}} + \Lambda_{\{1\}}\Lambda_{\{2k\}}\big).
\end{gather*}
The first $q$-commutator can be expanded by (\ref{IH 2}), the second by~(\ref{IH 4}), the third by Lemma~\ref{Corollary 1.3.4}. Writing everything down, a lot of common terms will cancel, eventually leading to
\begin{gather}
 [\Lambda_{\{2,4,6,\dots,2k,2k+1\}},\Lambda_{\{1,2k+1\}}]_q \nonumber\\
\quad = \big(q^{-2}-q^2\big)\Lambda_{\{1,2,4,\dots,2k\}} + \big(q-q^{-1}\big)\big(\Lambda_{\{2k+1\}}\Lambda_{\{1,2,4,\dots,2k,2k+1\}} + \Lambda_{\{1\}}\Lambda_{\{2,4,\dots,2k\}}\big).\!\!\!\label{first part of claim}
\end{gather}
This proves the first part of the claim.

By $\chi = 1^{\otimes (2k-1)}\otimes\Delta\otimes 1$ respectively $\chi = 1^{\otimes 2k}\otimes \tau_R$, (\ref{first part of claim}) implies that (\ref{standard relation}) holds for
\begin{gather}\label{IH 5}
A = \{2,4,\dots,2k,2k+1,2k+2\}, \qquad B = \{1,2k+2\}
\end{gather}
and
\begin{gather}\label{IH 6}
A = \{2,4,\dots,2k,2k+2\}, \qquad B = \{1,2k+2\}.
\end{gather}
Applying $\chi = \big(\tau_L\otimes 1^{\otimes 2k}\big)\cdots\big(\tau_L\otimes 1^{\otimes 2}\big)$ to (\ref{rank one 23-13}), we may write
\begin{gather}\label{expanded x}
\Lambda_{\{1,2k+1\}} = \frac{[\Lambda_{\{2k+1,2k+2\}},\Lambda_{\{1,2k+2\}}]_q}{q^{-2}-q^2} +\frac{\Lambda_{\{2k+2\}}\Lambda_{\{1,2k+1,2k+2\}} + \Lambda_{\{2k+1\}}\Lambda_{\{1\}}}{q+q^{-1}}.
\end{gather}
We have (\ref{commutation}) for $A = \{1,2k+2\}$ and $B = \{2,4,\dots,2k,2k+1\}$, as follows from Lemma \ref{Lemma 1.3.8} by $\chi = 1^{\otimes (2k-1)}\otimes\Delta\otimes 1$.
Substituting (\ref{expanded x}) in (\ref{to compute 1}) and using (\ref{q-anticomm 1}), (\ref{to compute 1}) becomes
\begin{gather*}
 \frac{[[\Lambda_{\{2,4,\dots,2k,2k+1\}},\Lambda_{\{2k+1,2k+2\}} ]_q,\Lambda_{\{1,2k+2\}} ]_q}{q^{-2}-q^2} \\
 \quad{}+ \frac{ \Lambda_{\{2k+2\}}[\Lambda_{\{2,4,\dots,2k,2k+1\}},\Lambda_{\{1,2k+1,2k+2\}}]_q +\big(q-q^{-1}\big)\Lambda_{\{1\}}\Lambda_{\{2k+1\}}\Lambda_{\{2,4,\dots,2k,2k+1\}}}{q+q^{-1}}.
\end{gather*}
The relation (\ref{standard relation}) holds for $A = \{2,4,\dots,2k,2k+1\}$ and $B = \{2k+1,2k+2\}$ by Lemma \ref{Corollary 1.3.4}. Hence (\ref{to compute 1}) becomes
\begin{gather}
 [\Lambda_{\{2,4,\dots,2k,2k+2\}},\Lambda_{\{1,2k+2\}}]_q \nonumber\\
{} - \frac{\Lambda_{\{2k+1\}} [\Lambda_{\{2,4,\dots,2k,2k+1,2k+2\}},\Lambda_{\{1,2k+2\}} ]_q+\big(q-q^{-1}\big)\Lambda_{\{2k+2\}}\Lambda_{\{2,4,\dots,2k\}}\Lambda_{\{1,2k+2\}}}{q+q^{-1}}\nonumber\\
{} +\frac{ \Lambda_{\{2k+2\}}[\Lambda_{\{2,4,\dots,2k,2k+1\}},\Lambda_{\{1,2k+1,2k+2\}}]_q +\big(q-q^{-1}\big)\Lambda_{\{1\}}\Lambda_{\{2k+1\}}\Lambda_{\{2,4,\dots,2k,2k+1\}}}{q+q^{-1}},\label{remaining commutator}
\end{gather}
where we have used (\ref{commutation}) for $A = \{2,4,\dots,2k\}$, $B = \{1,2k+2\}$, which follows from Lemma \ref{Lemma 1.3.8} by $\chi = 1^{\otimes 2k}\otimes\tau_R$. The first $q$-commutator can be expanded by (\ref{IH 6}), the second by (\ref{IH 5}). On the other hand, we already know an expression for (\ref{to compute 1}), namely (\ref{first part of claim}). Comparing these, the only remaining $q$-commutator in (\ref{remaining commutator}) can be expanded as
\begin{gather*}
 [\Lambda_{\{2,4,\dots,2k,2k+1\}},\Lambda_{\{1,2k+1,2k+2\}}]_q = \big(q^{-2}-q^2\big)\Lambda_{\{1,2,4,\dots,2k,2k+2\}}\\
 \quad{} + \big(q-q^{-1}\big)\big(\Lambda_{\{2k+1\}}\Lambda_{\{1,2,4,\dots,2k,2k+1,2k+2\}} + \Lambda_{\{1,2k+2\}}\Lambda_{\{2,4,\dots,2k\}}\big).
\end{gather*}
This concludes the induction.
\end{proof}

By a completely analogous inductive proof, one can show the following.

\begin{Lemma}The relation \eqref{standard relation} holds for the sets \eqref{sets 3} and \eqref{sets 4 without gap} with $k = 1$.
\end{Lemma}

Somewhat different is our strategy to prove the relation (\ref{standard relation}) for the sets (\ref{sets 4 without gap}) and (\ref{sets 4 with gap}).

\begin{Lemma}The standard relation \eqref{standard relation} holds for the sets \eqref{sets 4 without gap} and \eqref{sets 4 with gap}, i.e., for
\begin{gather*}
A = \{1,2,4,6,\dots,2k,2k+2\ell+2+\delta\}, \\ B = \{2,4,6,\dots,2k,2k+1+\delta,2k+3+\delta,\dots,2k+2\ell+1+\delta\},
\end{gather*}
with $k, \ell\in\N$ and $\delta\in\{0,1\}$.
\end{Lemma}
\begin{proof}We need to rewrite
\begin{gather}\label{to compute 2}
[\Lambda_{\{1,2,4,6,\dots,2k,2k+2\ell+2+\delta\}}, \Lambda_{\{2,4,6,\dots,2k,2k+1+\delta,2k+3+\delta,\dots,2k+2\ell+1+\delta\}}]_q.
\end{gather}
First observe that
\begin{gather}\label{comm rel yy}
[\Lambda_{\{1,2k+2\ell+3+\delta\}},\Lambda_{\{2,4,6,\dots,2k,2k+1+\delta,2k+3+\delta,\dots,2k+2\ell+1+\delta\}}] = 0.
\end{gather}
This follows from Lemma \ref{Lemma 1.3.8} with $k+\ell+\delta$ instead of $k$, by
\[\chi = \big(1^{\otimes (2k-1)}\otimes\Delta\otimes 1^{\otimes (2\ell+2)}\big)^{1-\delta}\big(1^{\otimes (2k+2\ell+2\delta)}\otimes\tau_R\big).\]
The relation (\ref{standard relation}) for (\ref{sets 2}) with $k+1$ instead of $k$, acted upon with
\[
\chi = \big(1^{\otimes 2k}\otimes\Delta\otimes 1^{\otimes (2\ell+1+\delta)}\big)\cdots\big(1^{\otimes 2k}\otimes\Delta\otimes 1^{\otimes 2}\big),
\]
gives rise to the identity
\begin{gather*}
 \Lambda_{\{1,2,4,6,\dots,2k,2k+2\ell+2+\delta\}} = \frac{[\Lambda_{\{2,4,6,\dots,2k,2k+2\ell+2+\delta,2k+2\ell+3+\delta\}},\Lambda_{\{1,2k+2\ell+3+\delta\}}]_q}{q^{-2}-q^2} \\ \qquad{} + \frac{\Lambda_{\{2k+2\ell+3+\delta\}}\Lambda_{\{1,2,4,\dots,2k,2k+2\ell+2+\delta,2k+2\ell+3+\delta\}} + \Lambda_{\{1\}}\Lambda_{\{2,4,\dots,2k,2k+2\ell+2+\delta\}}}{q+q^{-1}}.
\end{gather*}
Using (\ref{q-anticomm 3}) by (\ref{comm rel yy}), we may write (\ref{to compute 2}) as
\begin{gather}
 \frac{[[\Lambda_{\{2,4,\dots,2k,2k+2\ell+2+\delta,2k+2\ell+3+\delta\}},\Lambda_{\{2,4,\dots,2k,2k+1+\delta,2k+3+\delta,\dots,2k+2\ell+1+\delta\}} ]_q,\Lambda_{\{1,2k+2\ell+3+\delta\}}]_q}{q^{-2}-q^2} \nonumber\\
{} + \frac{ \Lambda_{\{2k+2\ell+3+\delta\}}[\Lambda_{\{1,2,4,\dots,2k,2k+2\ell+2+\delta,2k+2\ell+3+\delta\}}, \Lambda_{\{2,4,\dots,2k,2k+1+\delta,2k+3+\delta,\dots,2k+2\ell+1+\delta\}}]_q}{q+q^{-1}} \nonumber\\
{}+ \frac{\Lambda_{\{1\}}[\Lambda_{\{2,4,\dots,2k,2k+2\ell+2+\delta\}},\Lambda_{\{2,4,\dots,2k,2k+1+\delta,2k+3+\delta,\dots,2k+2\ell+1+\delta\}}]_q}{q+q^{-1}}.\label{long rel 1}
\end{gather}
The relation (\ref{standard relation}) is satisfied for
\begin{gather*}
A = \{2,4,6,\dots,2k,2k+2\ell+2+\delta\},\\ B = \{2,4,6,\dots,2k,2k+1+\delta,2k+3+\delta,\dots,2k+2\ell+1+\delta\},
\end{gather*}
as follows from (\ref{sets 3}) with $\ell+1$ instead of $k$, by
\[
\chi = \left(\overrightarrow{\prod_{n = 2\ell+2+\delta}^{2k+2\ell-2+\delta}}\big(\tau_L\otimes 1^{\otimes (n+1)}\big)\big(\Delta\otimes 1^{\otimes n}\big)\right)\big(\tau_L\otimes 1^{\otimes (2\ell+2)}\big)^{\delta}
\]
and putting a factor $1\otimes$ in front. By $\chi = 1^{\otimes (2k+2\ell+1+\delta)}\otimes\Delta$, we also have (\ref{standard relation}) for
\begin{gather*}
A = \{2,4,6,\dots,2k,2k+2\ell+2+\delta,2k+2\ell+3+\delta\}, \\ B = \{2,4,6,\dots,2k,2k+1+\delta,2k+3+\delta,\dots,2k+2\ell+1+\delta\}.
\end{gather*}
This helps us expand the first and third line of (\ref{long rel 1}), such that (\ref{to compute 2}) becomes
\begin{gather}
 [\Lambda_{\{2k+1+\delta,2k+3+\delta,\dots,2k+2\ell+1+\delta,2k+2\ell+2+\delta,2k+2\ell+3+\delta\}},\Lambda_{\{1,2k+2\ell+3+\delta\}}]_q \nonumber\\
{}-\frac{\Lambda_{\{2,4\dots,2k\}}}{q+q^{-1}} [\Lambda_{\{2,4,\dots,2k,2k+1+\delta,2k+3+\delta,\dots,2k+2\ell+1+\delta,2k+2\ell+2+\delta,2k+2\ell+3+\delta\}},\Lambda_{\{1,2k+2\ell+3+\delta\}}]_q \nonumber\\
{} -\frac{\Lambda_{\{2k+1+\delta,2k+3+\delta,\dots,2k+2\ell+1+\delta\}}}{q+q^{-1}} [\Lambda_{\{2k+2\ell+2+\delta,2k+2\ell+3+\delta\}},\Lambda_{\{1,2k+2\ell+3+\delta\}}]_q \nonumber\\
 {} + \frac{\Lambda_{\{2k+2\ell+3+\delta\}}}{q+q^{-1}} [\Lambda_{\{1,2,4,\dots,2k,2k+2\ell+2+\delta,2k+2\ell+3+\delta\}}, \Lambda_{\{2,4,\dots,2k,2k+1+\delta,2k+3+\delta,\dots,2k+2\ell+1+\delta\}}]_q \nonumber\\
{} - \big(q-q^{-1}\big)\Lambda_{\{1\}}\Lambda_{\{2k+1+\delta,2k+3+\delta,\dots,2k+2\ell+1+\delta,2k+2\ell+2+\delta\}} \nonumber\\
{} + \frac{q-q^{-1}}{q+q^{-1}}\Lambda_{\{1\}}\Lambda_{\{2,4,\dots,2k\}}\Lambda_{\{2,4,\dots,2k,2k+1+\delta,2k+3+\delta,\dots,2k+2\ell+1+\delta,2k+2\ell+2+\delta\}} \nonumber\\
{} + \frac{q-q^{-1}}{q+q^{-1}} \Lambda_{\{1\}}\Lambda_{\{2k+2\ell+2+\delta\}}\Lambda_{\{2k+1+\delta,2k+3+\delta,\dots,2k+2\ell+1+\delta\}}.\label{long expr 1}
\end{gather}
Here we have used the commutation relations
\begin{gather*}
[\Lambda_{\{2,4,\dots,2k\}},\Lambda_{\{1,2k+2\ell+3+\delta\}}] = 0, \\
[\Lambda_{\{2,4,\dots,2k\}},\Lambda_{\{2,4,\dots,2k,2k+1+\delta,2k+3+\delta,\dots,2k+2\ell+1+\delta,2k+2\ell+2+\delta,2k+2\ell+3+\delta\}}] = 0, \\
[\Lambda_{\{2k+1+\delta,2k+3+\delta,\dots,2k+2\ell+1+\delta\}},\Lambda_{\{2k+2\ell+2+\delta,2k+2\ell+3+\delta\}}] = 0, \\
[\Lambda_{\{2k+1+\delta,2k+3+\delta,\dots,2k+2\ell+1+\delta\}},\Lambda_{\{1,2k+2\ell+3+\delta\}}] = 0,
\end{gather*}
which follow from Lemmas \ref{Lemma 1.3.7} and \ref{Lemma 1.3.8}.

The relation (\ref{standard relation}) holds for the sets
\begin{gather}
\begin{split}&
A = \{2k+1+\delta,2k+3+\delta,\dots,2k+2\ell+1+\delta,2k+2\ell+2+\delta,2k+2\ell+3+\delta\},  \\
& B = \{1,2k+2\ell+3+\delta\};\end{split} \label{int sets 1}\\
\label{int sets 2}
\begin{split}&
A = \{2k+2\ell+2+\delta,2k+2\ell+3+\delta\},\\
& B = \{1,2k+2\ell+3+\delta\};\end{split} \vspace{2mm} \\
\begin{split}&
A = \{2,4,\dots,2k,2k+1+\delta,2k+3+\delta,\dots, \\
&\hphantom{A = \{}{} 2k+2\ell+1+\delta,2k+2\ell+2+\delta,2k+2\ell+3+\delta\}, \\
& B = \{1,2k+2\ell+3+\delta\}.\end{split}\label{int sets 3}
\end{gather}
For (\ref{int sets 1}), this follows from (\ref{sets 2}) with $\ell+1$ instead of $k$, upon applying
\[
\chi = \left(\overrightarrow{\prod_{n = 2\ell+3}^{2k+2\ell+1+\delta}}\big(\tau_L\otimes 1^{\otimes n}\big)\right)\big(1^{\otimes (2\ell+1)}\otimes\Delta\otimes1\big).
\]
The claim for (\ref{int sets 2}) follows from (\ref{rank one 23-13}), by $\chi = \overrightarrow{\prod\limits_{n = 2}^{2k+2\ell+1+\delta}}\big(\tau_L\otimes 1^{\otimes n}\big)$,
and finally (\ref{int sets 3}) follows from (\ref{sets 2}) with $k+\ell+\delta$ instead of $k$, by
$\chi = \big(1^{\otimes (2k-1)}\otimes \Delta\otimes 1^{\otimes (2\ell+2)}\big)^{1-\delta}\big(1^{\otimes (2k+2\ell-1+2\delta)}\otimes\Delta\otimes 1\big)$.
Bringing the fourth line of (\ref{long expr 1}) to the left-hand side and expanding everything as explained, we find
\begin{gather*}
[\Lambda_{\{1,2,4,6,\dots,2k,2k+2\ell+2+\delta\}}, \Lambda_{\{2,4,6,\dots,2k,2k+1+\delta,2k+3+\delta,\dots,2k+2\ell+1+\delta\}}]_q\\
- \frac{\Lambda_{\{2k+2\ell+3+\delta\}}}{q+q^{-1}}[\Lambda_{\{1,2,4,\dots,2k,2k+2\ell+2+\delta,2k+2\ell+3+\delta\}}, \Lambda_{\{2,4,\dots,2k,2k+1+\delta,2k+3+\delta,\dots,2k+2\ell+1+\delta\}}]_q \\
= \big(q^{-2}-q^2\big)\Lambda_{\{1,2k+1+\delta,2k+3+\delta,\dots,2k+2\ell+1+\delta,2k+2\ell+2+\delta\}}\\
+ \big(q-q^{-1}\big)\Lambda_{\{2,4,\dots,2k\}}\Lambda_{\{1,2,4,\dots,2k,2k+1+\delta,2k+3+\delta,\dots,2k+2\ell+1+\delta,2k+2\ell+2+\delta\}}\\
 + \big(q-q^{-1}\big) \Lambda_{\{1,2k+2\ell+2+\delta\}}\Lambda_{\{2k+1+\delta,2k+3+\delta,\dots,2k+2\ell+1+\delta\}} \\
 + \big(q-q^{-1}\big)\Lambda_{\{2k+2\ell+3+\delta\}}\Lambda_{\{1,2k+1+\delta,2k+3+\delta,\dots,2k+2\ell+1+\delta,2k+2\ell+2+\delta,2k+2\ell+3+\delta\}} \\ -\frac{q-q^{-1}}{q+q^{-1}}\Lambda_{\{2k+2\ell+3+\delta\}}\Lambda_{\{2,4,\dots,2k\}}\Lambda_{\{1,2,4,\dots,2k,2k+1+\delta,2k+3+\delta,\dots,2k+2\ell+1 +\delta,2k+2\ell+2+\delta,2k+2\ell+3+\delta\}} \\
 -\frac{q-q^{-1}}{q+q^{-1}}\Lambda_{\{2k+2\ell+3+\delta\}} \Lambda_{\{1,2k+2\ell+2+\delta,2k+2\ell+3+\delta\}}\Lambda_{\{2k+1+\delta,2k+3+\delta,\dots,2k+2\ell+1+\delta\}}.
\end{gather*}
Now let us denote by $\Theta$ and $\Xi$ respectively the expressions
\begin{gather*}
\Theta = [\Lambda_{\{1,2,4,6,\dots,2k,2k+2\ell+2+\delta\}}, \Lambda_{\{2,4,6,\dots,2k,2k+1+\delta,2k+3+\delta,\dots,2k+2\ell+1+\delta\}}]_q, \\
\Xi =\big(q^{-2}-q^2\big)\Lambda_{\{1,2k+1+\delta,2k+3+\delta,\dots,2k+2\ell+1+\delta,2k+2\ell+2+\delta\}}\\
\hphantom{\Xi =}{} + \big(q-q^{-1}\big)\Lambda_{\{2,4,\dots,2k\}}\Lambda_{\{1,2,4,\dots,2k,2k+1+\delta,2k+3+\delta,\dots,2k+2\ell+1+\delta,2k+2\ell+2+\delta\}}\\
\hphantom{\Xi =}{} + \big(q-q^{-1}\big) \Lambda_{\{1,2k+2\ell+2+\delta\}}\Lambda_{\{2k+1+\delta,2k+3+\delta,\dots,2k+2\ell+1+\delta\}},
\end{gather*}
both in the $(2k+2\ell+2+\delta)$-fold tensor product. Then the relation above becomes
\begin{gather*}
 \Theta \otimes 1 -\frac{1}{q+q^{-1}}\Lambda_{\{2k+2\ell+3+\delta\}} \big(1^{\otimes (2k+2\ell+1+\delta)}\otimes\Delta\big) \Theta\\
 \qquad{}= \Xi\otimes 1 -\frac{1}{q+q^{-1}}\Lambda_{\{2k+2\ell+3+\delta\}} \big(1^{\otimes (2k+2\ell+1+\delta)}\otimes\Delta\big) \Xi,
\end{gather*}
or equivalently
\begin{gather}\label{eq A-B}
(\Theta-\Xi)\otimes 1 = \frac{1}{q+q^{-1}}\Lambda_{\{2k+2\ell+3+\delta\}} \big(1^{\otimes (2k+2\ell+1+\delta)}\otimes\Delta\big)(\Theta-\Xi).
\end{gather}
We want to show that $\Theta = \Xi$.

Let us start by writing $\Theta-\Xi$ explicitly as a tensor product, i.e.,
\[
\Theta-\Xi = \sum_{i\in I}a_1^{(i)}\otimes\cdots\otimes a_{2k+2\ell+2+\delta}^{(i)},
\]
for certain $a_j^{(i)}\in U_q(\mathfrak{sl}_2)$ and some finite index set $I$, and where we have grouped the elements in such a way that the set
\begin{gather}\label{convention on indices}
\big\{a_1^{(i)}\otimes\dots\otimes a_{2k+2\ell+1+\delta}^{(i)}\colon i\in I\big\}
\end{gather}
is linearly independent. The equality (\ref{eq A-B}) then asserts
\[
\sum_{i\in I}a_1^{(i)}\otimes\dots\otimes a_{2k+2\ell+1+\delta}^{(i)}\otimes\left(a_{2k+2\ell+2+\delta}^{(i)}\otimes 1 - \frac{1}{q+q^{-1}} (1\otimes\Lambda )\Delta\big(a_{2k+2\ell+2+\delta}^{(i)}\big) \right)=0,
\]
which by our convention (\ref{convention on indices}) implies
\[
a_{2k+2\ell+2+\delta}^{(i)}\otimes 1 = \frac{1}{q+q^{-1}} (1\otimes\Lambda )\Delta\big(a_{2k+2\ell+2+\delta}^{(i)}\big)
\]
for every $i\in I$. But this now implies that every $a_{2k+2\ell+2+\delta}^{(i)}$ vanishes, since if not, then $\Delta\big(a_{2k+2\ell+2+\delta}^{(i)}\big) = \sum\limits_{j\in J}b_1^{(j)}\otimes b_2^{(j)}$ is nontrivial and must be such that
\[
a_{2k+2\ell+2+\delta}^{(i)}\otimes 1 = \frac{1}{q+q^{-1}}\sum_{j\in J}b_1^{(j)}\otimes \Lambda b_2^{(j)}.
\]
But by comparing degrees in the generators $E$ and $F$, it is clear that $\Lambda b_2^{(j)}\neq 1$ for every possible $b_2^{(j)}\in U_q(\mathfrak{sl}_2)$, i.e., $\Lambda$ has no inverse. So we have $a_{2k+2\ell+2+\delta}^{(i)} = 0$ for every $i$. Hence we have shown that indeed $\Theta = \Xi$. This concludes the proof.
\end{proof}

The two preceding proofs manifest the general strategy that can be used to show Proposition~\ref{prop - fundamental relations}. Therefore, and for the sake of brevity, we will only concisely mention how one can show the remaining relations.

\begin{Lemma}The standard relation \eqref{standard relation} holds for the sets~\eqref{sets 1}.
\end{Lemma}
\begin{proof}[Sketch of the proof] Use (\ref{standard relation}) for $A = \{2,4,\dots,2k,2k+2\}$, $B = \{1,2k+2\}$, as follows from~(\ref{sets 2}), to rewrite $\Lambda_{\{1,2,4,\dots,2k\}}$ in
\[
[\Lambda_{\{1,2,4,\dots,2k\}},\Lambda_{\{2,4,\dots,2k,2k+1\}}]_q.
\]
Use (\ref{q-anticomm 3}) to switch the order of the nested $q$-commutators. Use (\ref{sets 4 without gap}) with $\ell = 0$ and (\ref{rank one 13-12}) acted upon with $\chi = \overrightarrow{\prod\limits_{\ell = 1}^{k-1}}\big(\tau_L\otimes 1^{\otimes (2\ell+1)}\big)\big(\Delta\otimes 1^{\otimes 2\ell}\big)$ to expand further. The remaining $q$-commutators can be worked out from~(\ref{rank one 23-13}) and~(\ref{sets 2}).
\end{proof}

\begin{Lemma}The standard relation \eqref{standard relation} holds for the sets \eqref{sets 5 without gap} and \eqref{sets 5 with gap}.
\end{Lemma}
\begin{proof}[Sketch of the proof] Use (\ref{standard relation}) for (\ref{sets 3}) with $\ell+1$ instead of $k$, acted upon with \[\chi = \big(\tau_L\otimes 1^{\otimes (2k+2\ell+\delta)}\big)\cdots\big(\tau_L\otimes 1^{\otimes (2\ell+2)}\big),\] to rewrite $\Lambda_{B'}$ in $[\Lambda_{A'},\Lambda_{B'}]_q$ with
\begin{gather*}
A' = \{2,4,\dots,2k,2k+1+\delta,2k+3+\delta,\dots,2k+2\ell+1+\delta\}, \\ B' = \{2k+1+\delta,2k+3+\delta,\dots,2k+2\ell+1+\delta,2k+2\ell+2+\delta\},
\end{gather*}
where $\delta\in\{0,1\}$. Use (\ref{q-anticomm 2}) and expand the inner $q$-commutator by (\ref{sets 2}) acted upon with
\begin{gather*}
\chi = \big(1^{\otimes (2k+2\ell-1+\delta)}\otimes\tau_R\big)\big(1^{\otimes (2k+2\ell-2+\delta)}\otimes\Delta\big)\cdots\\
\hphantom{\chi =}{}\times \big(1^{\otimes(2k+1+\delta)}\otimes\tau_R\big) \big(1^{\otimes (2k+\delta)}\otimes\Delta\big)\big(1^{\otimes 2k}\otimes\tau_R\big)^{\delta}.
\end{gather*}
Apply the relation (\ref{standard relation}) for the sets
\begin{gather*}
\begin{split}
& A = \{1,2k+2\ell+2+\delta\}, \\
&  B = \{1,2,4,\dots,2k\}; \end{split}\vspace{2mm} \\
\begin{split}
& A = \{1,2k+2\ell+2+\delta\},\\
& B = \{1,2,4,\dots,2k,2k+1+\delta,2k+3+\delta,\dots,2k+2\ell+1+\delta \},\end{split}
\end{gather*}
as follows from (\ref{sets 3}) by $\chi = \big(1^{\otimes (2k+2\ell+\delta)}\otimes\tau_R\big)\cdots\big(1^{\otimes 2k}\otimes\tau_R\big)$ respectively from~(\ref{sets 3}) with $k+\ell+\delta$ instead of~$k$, acted upon with $\chi = \big(1^{\otimes (2k-1)}\otimes\Delta\otimes 1^{\otimes (2\ell+1)}\big)^{1-\delta}$. In our expression for $[\Lambda_{A'},\Lambda_{B'}]_q$, there is now one term left containing a $q$-commutator, namely
\begin{gather}
\frac{1}{q+q^{-1}}\Lambda_{\{1\}}[\Lambda_{\{2,4,\dots,2k,2k+1+\delta,2k+3+\delta,\dots,2k+2\ell+1+\delta\}},\nonumber\\
\hphantom{\frac{1}{q+q^{-1}}\Lambda_{\{1\}}[}{} \Lambda_{\{1,2k+1+\delta,2k+3+\delta,\dots,2k+2\ell+1+\delta,2k+2\ell+2+\delta\}}]_q.\label{left to work out}
\end{gather}
On the other hand, we already know how to expand $[\Lambda_{A'},\Lambda_{B'}]_q$ from (\ref{sets 1}) with $\ell+1$ instead of~$k$, acted upon with
\begin{gather*}
\chi = \big(\tau_L\otimes 1^{\otimes (2k+2\ell-1+\delta)}\big)\big(\Delta\otimes 1^{\otimes (2k+2\ell-2+\delta)}\big)\cdots\\
\hphantom{\chi =}{}\times \big(\tau_L\otimes 1^{\otimes (2\ell+3+\delta)}\big)\big(\Delta\otimes 1^{\otimes (2\ell+2+\delta)}\big)\big(\tau_L\otimes 1^{\otimes (2\ell+2)}\big)^{\delta},
\end{gather*}
and adding a factor $1\otimes$ in front. This leads to the sought expression for the remaining $q$-commutator in~(\ref{left to work out}).
\end{proof}

To show the final cases (\ref{sets 6 without gap}) and (\ref{sets 6 with gap}), we will need another commutation relation.

\begin{Lemma}\label{Lemma 1.4.7} For any $k\in\N$ one has
\[
[\Lambda_{\{1,2k+1\}},\Lambda_{\{1,2,4,\dots,2k,2k+1\}}] = 0.
\]
\end{Lemma}
\begin{proof}By induction on $k$, the case $k = 1$ being obvious from a direct calculation. Suppose the claim has been proven for $k-1$. Since~(\ref{standard relation}) holds for~(\ref{sets 5 without gap}), with $k-1$ instead of $k$ and with $\ell = 0$, we have by $\chi = 1^{\otimes (2k-1)}\otimes\Delta$
\begin{gather*}
\Lambda_{\{1,2,4,\dots,2k,2k+1\}} = \frac{[\Lambda_{\{2,4,\dots,2k-2,2k-1\}},\Lambda_{\{1,2k-1,2k,2k+1\}}]_q}{q^{-2}-q^2} \\
\hphantom{\Lambda_{\{1,2,4,\dots,2k,2k+1\}} =}{} +\frac{\Lambda_{\{2k-1\}}\Lambda_{\{1,2,4,\dots,2k-2,2k-1,2k,2k+1\}} + \Lambda_{\{2,4,\dots,2k-2\}}\Lambda_{\{1,2k,2k+1\}}}{q+q^{-1}}.
\end{gather*}
It suffices now to show that each of the generators in the right-hand side commutes with $\Lambda_{\{1,2k+1\}}$. For $\Lambda_{\{2,4,\dots,2k-2,2k-1\}}$ and $\Lambda_{\{2,4,\dots,2k-2\}}$ this follows from Lemma~\ref{Lemma 1.3.8} with $k-1$ instead of $k$, upon applying $\chi = \big(1^{\otimes(2k-1)}\otimes \tau_R\big)\big(1^{\otimes(2k-3)}\otimes\Delta\otimes 1\big)$ and $\chi = \big(1^{\otimes(2k-1)}\otimes \tau_R\big)\big(1^{\otimes(2k-2)}\otimes \tau_R\big)$ respectively. For $\Lambda_{\{1,2,4,\dots,2k-2,2k-1,2k,2k+1\}}$ we apply $\chi = \big(1^{\otimes (2k-3)}\otimes\Delta\otimes 1\otimes 1\big)\big(1^{\otimes (2k-3)}\otimes\Delta\otimes 1\big)$ to the induction hypothesis. The remaining nontrivial commutation relations follow from $[\Lambda_{\{1,3\}},\Lambda_{\{1,2,3\}}] = 0$ upon applying respectively
\[
\chi = \big(\tau_L\otimes 1^{\otimes (2k-1)}\big)\cdots\big(\tau_L\otimes 1^{\otimes 4}\big)\big(\tau_L\otimes 1^{\otimes 3}\big)(1\otimes\Delta\otimes 1)
\]
and
\begin{gather*}
\chi = \big(\tau_L\otimes 1^{\otimes (2k-1)}\big)\cdots\big(\tau_L\otimes 1^{\otimes 3}\big)\big(\tau_L\otimes 1^{\otimes 2}\big).\tag*{\qed}
\end{gather*}\renewcommand{\qed}{}
\end{proof}

\begin{Corollary}\label{Corollary 1.4.8}$\Lambda_{\{1,2k\}}$ commutes with $\Lambda_{\{1,2,4,\dots,2k\}}$ and $\Lambda_{\{1,3,5,\dots,2k-1,2k\}}$.
\end{Corollary}
\begin{proof}Act on the relation of Lemma \ref{Lemma 1.4.7}, with $k-1$ instead of $k$, with $\chi = 1^{\otimes (2k-2)}\otimes\tau_R$ and $\chi = \tau_L\otimes 1^{\otimes (2k-2)}$ respectively.
\end{proof}

\begin{Lemma}\label{lemma sets 6}The standard relation \eqref{standard relation} holds for the sets \eqref{sets 6 without gap} and \eqref{sets 6 with gap}, i.e., for
\begin{gather}
A = \{1,2k+1+\delta,2k+3+\delta,\dots,2k+2\ell+1+\delta,2k+2\ell+2+\delta\}, \nonumber\\
 B = \{1,2,4,\dots,2k,2k+2\ell+2+\delta\}\label{to compute 5}
\end{gather}
with $k,\ell\in\N$ and $\delta \in\{0,1\}$.
\end{Lemma}

\begin{proof}[Sketch of the proof] By induction on $k$. Already the induction basis $k = 1$, i.e., to prove (\ref{standard relation}) for
\begin{gather}\label{induction basis on k}
A = \{1,3+\delta,5+\delta,\dots,2\ell+3+\delta,2\ell+4+\delta\}, \qquad B = \{1,2,2\ell+4+\delta\},
\end{gather}
is difficult and needs to be shown by induction on $\ell$, the case $\ell= 0$ being trivial. Suppose thus that the claim holds for $\ell-1$, i.e., the relation~(\ref{standard relation}) holds for
\begin{gather}\label{IIH 1}
A = \{1,3+\delta,5+\delta,\dots,2\ell+1+\delta,2\ell+2+\delta\}, \qquad B = \{1,2,2\ell+2+\delta\}.
\end{gather}
We will find an expression for
\begin{gather}\label{to compute 4}
[\Lambda_{\{1,3+\delta,5+\delta,\dots,2\ell+3+\delta,2\ell+4+\delta\}},\Lambda_{\{1,2,2\ell+4+\delta\}}]_q.
\end{gather}
Rewrite $\Lambda_{\{1,3+\delta,5+\delta,\dots,2\ell+3+\delta,2\ell+4+\delta\}}$ in (\ref{to compute 4}) by (\ref{sets 4 without gap}) with $k = 1$ and $\ell = 0$, acted upon by \begin{gather*}\begin{split}&
\chi = \big(\tau_L\otimes 1^{\otimes (2\ell+3)}\big)^{\delta}\big(\tau_L\otimes 1^{\otimes (2\ell+2)}\big)\big(\Delta\otimes 1^{\otimes (2\ell+1)}\big)\cdots\\
& \hphantom{\chi =}{}\times \big(\tau_L\otimes 1^{\otimes 6}\big)\big(\Delta\otimes 1^{\otimes 5}\big)\big(\tau_L\otimes 1^{\otimes 4}\big)\big(\Delta\otimes 1^{\otimes 3}\big).
\end{split}
\end{gather*}
Use (\ref{q-anticomm 3}) and expand further by (\ref{standard relation}) for
\begin{gather*}\begin{split}&
A = \{1,3+\delta,5+\delta,\dots,2\ell+1+\delta,2\ell+2+\delta,2\ell+4+\delta\}, \\
& B = \{1,2,2\ell+4+\delta\}; \end{split}\vspace{2mm}\\
\begin{split}&
A = \{1,3+\delta,5+\delta,\dots,2\ell+1+\delta,2\ell+2+\delta,2\ell+3+\delta,2\ell+4+\delta\}, \\
& B = \{1,2,2\ell+4+\delta\}; \end{split}\vspace{2mm}\\
\begin{split}&
A = \{1,3+\delta,5+\delta,\dots,2\ell+1+\delta,2\ell+4+\delta\}, \\
& B= \{1,2,2\ell+4+\delta\},\end{split}
\end{gather*}
as follows from the induction hypothesis (\ref{IIH 1}) by respectively
\begin{gather*}
\chi = \big(1^{\otimes (2\ell+\delta)}\otimes\Delta\otimes 1^{\otimes 2}\big)\big(1^{\otimes (2\ell+1+\delta)}\otimes\tau_R\big), \\
\chi = \big(1^{\otimes (2\ell+\delta)}\otimes\Delta\otimes 1^{\otimes 2}\big)\big(1^{\otimes (2\ell+\delta)}\otimes\Delta\otimes 1\big), \\
\chi = \big(1^{\otimes (2\ell+2+\delta)}\otimes\tau_R\big)\big(1^{\otimes (2\ell+1+\delta)}\otimes\tau_R\big).
\end{gather*}
Use the relation $[\Lambda_{\{1,2\ell+4+\delta\}},\Lambda_{\{1,2,3+\delta,5+\delta,\dots,2\ell+1+\delta,2\ell+2+\delta,2\ell+4+\delta\}}] = 0$, as follows from Lem\-ma~\ref{Lemma 1.4.7} with $\ell+\delta$ instead of $k$, after applying
\[
\chi =\big(1\otimes\Delta\otimes 1^{\otimes (2\ell+1)}\big)^{1-\delta}\big(1^{\otimes (2\ell-1+2\delta)}\otimes\Delta\otimes 1^{\otimes 2}\big)\big(1^{\otimes (2\ell+2\delta)}\otimes\tau_R\big).
\]
The remaining $q$-commutators can now be expanded using Lemma \ref{Corollary 1.3.4} and (\ref{sets 4 without gap}) with $k = 1$ and $\ell = 0$, acted upon by
\begin{gather*}
\chi = \big(\Delta\otimes 1^{\otimes (2\ell+2+\delta)}\big)\big(\tau_L\otimes 1^{\otimes (2\ell+2)}\big)^{\delta}\big(\Delta\otimes 1^{\otimes (2\ell+1)}\big)\big(\tau_L\otimes 1^{\otimes 2\ell}\big)\big(\Delta\otimes 1^{\otimes (2\ell-1)}\big)\cdots\\
\hphantom{\chi =}{}\times \big(\tau_L\otimes 1^{\otimes 4}\big)\big(\Delta\otimes 1^{\otimes 3}\big).
\end{gather*}
This leads to (\ref{standard relation}) for (\ref{induction basis on k}), which will serve as the basis for the induction on $k$ we perform in our global proof, i.e., to show (\ref{standard relation}) for (\ref{to compute 5}).

Suppose now the claim holds for $k-1$, i.e., the relation (\ref{standard relation}) holds for
\begin{gather}
A = \{1,2k-1+\delta,2k+1+\delta,\dots,2k+2\ell-1+\delta,2k+2\ell+\delta\}, \nonumber\\
 B = \{1,2,4, \dots,2k-2,2k+2\ell+\delta\},\label{iih 1}
\end{gather}
for arbitrary $\ell\in\N$. We will find an expression for
\begin{gather}
\label{to compute 6}
[\Lambda_{\{1,2k+1+\delta,2k+3+\delta,\dots,2k+2\ell+1+\delta,2k+2\ell+2+\delta\}},\Lambda_{\{1,2,4,\dots,2k,2k+2\ell+2+\delta\}}]_q.
\end{gather}
Use (\ref{sets 5 without gap}) with $k = 1$ and $\ell = 0$, acted upon with
\[
\chi = \left(\overrightarrow{\prod_{m = 2k+1}^{2k+2\ell+\delta}}\big(1^{\otimes m}\otimes\tau_R\big)\right)\left(\overrightarrow{\prod_{m = 2}^k}\big(1^{\otimes 2m}\otimes\tau_R\big)\big(1^{\otimes (2m-1)}\otimes\Delta\big)\right),
\]
to rewrite $\Lambda_{\{1,2,4,\dots,2k,2k+2\ell+2+\delta\}}$ in (\ref{to compute 6}) and apply~(\ref{q-anticomm 2}). Expand further by~(\ref{standard relation}) for $A = \{1,2k+1+\delta, 2k+3+\delta,\dots,2k+2\ell+1+\delta,2k+2\ell+2+\delta\}$ and
\begin{gather*}
B = \{1,3,4,6,\dots,2k-2,2k, 2k+2\ell+2+\delta\}, \\
B = \{1,2,3,4,6,\dots,2k-2,2k,2k+2\ell+2+\delta\}, \\
B = \{1,4,6,\dots,2k-2,2k,2k+2\ell+2+\delta\},
\end{gather*}
as follows from the induction hypothesis (\ref{iih 1}) by applying respectively
\begin{gather*}
\chi = \big(\tau_L\otimes 1^{\otimes (2k+2\ell+\delta)}\big)\big(1\otimes\Delta\otimes 1^{\otimes (2k+2\ell-2+\delta)}\big), \\
\chi = \big(1\otimes\Delta\otimes 1^{\otimes (2k+2\ell-1+\delta)}\big)\big(1\otimes\Delta\otimes 1^{\otimes (2k+2\ell-2+\delta)}\big), \\
\chi = \big(\tau_L\otimes 1^{\otimes (2k+2\ell+\delta)}\big)\big(\tau_L\otimes 1^{\otimes (2k+2\ell-1+\delta)}\big).
\end{gather*}
Apply the commutation relation \[[\Lambda_{\{1,2k+2\ell+2+\delta\}},\Lambda_{\{1,3,4,6,\dots,2k,2k+1+\delta,2k+3+\delta,\dots,2k+2\ell+1+\delta,2k+2\ell+2+\delta\}}] = 0,\] which follows from Lem\-ma~\ref{Lemma 1.4.7}. The remaining $q$-commutators can now be expanded using Lem\-ma~\ref{Corollary 1.3.4} and~(\ref{sets 5 without gap}) with $k = 1$ and $\ell = 0$. This leads to the anticipated expression for~(\ref{to compute 6}).
\end{proof}

\section{More commutation relations}\label{section - more commutation relations}
In this section we derive several criteria on the sets $A$ and $B$ for the elements $\Lambda_A$ and $\Lambda_B$ to commute. Eventually, it will be our aim to show that it suffices that $B\subseteq A$ in order to have $[\Lambda_A,\Lambda_B] = 0$. This result, stated in Theorem \ref{thm - commutation general}, will be proven in Section \ref{section - proof commutation general} and its proof will rely on the results of the present section.

\begin{Lemma}\label{Lemmas 1.3.9 to 1.3.12} The element $\Lambda_{\{1,3,5,\dots,2k-1\}}$ commutes with $\Lambda_{[1;2k]}$ and $\Lambda_{[1;2k-1]}$, and $\Lambda_{\{2,4,6,\dots,2k\}}$ commutes with $\Lambda_{[1;2k]}$ and $\Lambda_{[1;2k+1]}$.
\end{Lemma}
\begin{proof}We will only show the first relation, the others follow in complete analogy. We proceed by induction on $k$, the case $k = 1$ being trivial. Suppose hence that the claim holds for $k-1$, i.e., $\Lambda_A$ commutes with $\Lambda_B$ for
\begin{gather}\label{ih case 1}
A = \{1,3,5,\dots,2k-3\}, \qquad B = [1;2k-2].
\end{gather}
By Lemma \ref{Corollary 1.3.4} we may write
\begin{gather*}
\Lambda_{\{1,3,5,\dots,2k-1\}} = \frac{[\Lambda_{\{1,3,5,\dots,2k-3,2k-2\}},\Lambda_{\{2k-2,2k-1\}}]_q}{q^{-2}-q^2} \\
\hphantom{\Lambda_{\{1,3,5,\dots,2k-1\}} =}{} +\frac{\Lambda_{\{2k-2\}}\Lambda_{\{1,3,5,\dots,2k-3,2k-2,2k-1\}} + \Lambda_{\{2k-1\}}\Lambda_{\{1,3,5,\dots,2k-3\}}}{q+q^{-1}}.
\end{gather*}
Hence it suffices to show that $\Lambda_{[1;2k]}$ commutes with each of the terms in the right-hand side. For $\Lambda_{\{1,3,5,\dots,2k-3,2k-2\}} $, $\Lambda_{\{1,3,5,\dots,2k-3,2k-2,2k-1\}}$ and $\Lambda_{\{1,3,5,\dots,2k-3\}} $ this follows from the induction hypothesis~(\ref{ih case 1}) by $\chi = \big(1^{\otimes (2k-2)}\otimes\Delta\big)\big(1^{\otimes (2k-4)}\otimes\Delta\otimes 1\big)$, $\chi = \big(1^{\otimes (2k-3)}\otimes\Delta\otimes 1\big)\big(1^{\otimes (2k-4)}\otimes\Delta\otimes 1)$ and $\chi = \big(1^{\otimes (2k-2)}\otimes\Delta\big)\big(1^{\otimes (2k-3)}\otimes\Delta\big)$ respectively. For $\Lambda_{\{2k-2,2k-1\}}$ this follows from $[\Lambda_{\{2\}},\Lambda_{\{1,2,3\}}] = 0$ acted upon with $1\otimes\Delta\otimes 1$ and iterations of $\Delta\otimes 1^{\otimes n}$.
\end{proof}

Another useful commutation relation relies on Lemma \ref{lemma sets 6}. The proof is similar to the one above, we will hence just sketch it.

\begin{Lemma}\label{Corollary 1.4.10}The element $\Lambda_{\{2,4,\dots,2k\}}$ commutes with $\Lambda_{\{1,2,4,\dots,2k,2k+1\}}$.
\end{Lemma}
\begin{proof}[Sketch of the proof] By induction on $k$. Rewrite $\Lambda_{\{2,4,\dots,2k\}}$ by~(\ref{standard relation}) for~(\ref{sets 6 with gap}), with $k-1$ instead of $k$ and $\ell = 0$. Check that $\Lambda_{\{1,2,4,\dots,2k,2k+1\}}$ commutes with all generators in the new expression. This requires us to use Lemma~\ref{Lemma 1.4.7}, the first statement of Corollary~\ref{Corollary 1.4.8} acted upon with $1^{\otimes (2k-1)}\otimes\Delta$, the induction hypothesis acted upon with $\big(1^{\otimes (2k-1)}\otimes\Delta\big)\big(1^{\otimes (2k-2)}\otimes\tau_R\big)$ and the relation $[\Lambda_{\{1,3\}},\Lambda_{\{1,2,3\}}] = 0$ acted upon with iterations of $ \big(\Delta\otimes 1^{\otimes (n+1)}\big)\big(\tau_L\otimes 1^{\otimes n}\big)$.
\end{proof}

In the following commutation relations, the indexing subsets will no longer be defined by integers $k$ and $\ell$, but rather by more general conditions.

\begin{Lemma}\label{Proposition 1.5.2}Let $A$ be a set of consecutive integers and $B\subseteq A$. Then $\Lambda_A$ and $\Lambda_B$ commute.
\end{Lemma}
\begin{proof}Without loss of generality, we may assume $A = [1;n]$ and write $B$ as
\[
B = [i_1;j_1]\cup\cdots\cup[i_k;j_k].
\]
The claim follows upon applying
\[
\left(\overrightarrow{\prod_{\ell=n-i_1+1}^{n-2}}\big(\Delta\otimes 1^{\otimes \ell}\big)\right)^{1-\delta_{i_1,1}}\left(\overrightarrow{\prod_{\ell=j_k-i_1+2-\delta_{i_1,1}}^{n-i_1-\delta_{i_1,1}}}\big(1^{\otimes \ell}\otimes\Delta\big)\right)^{1-\delta_{j_k,n}}\chi',
\]
with
\[
\chi' = \overleftarrow{\prod_{m=1-\delta_{i_1,1}}^{2k-1-\delta_{i_1,1}}}\overrightarrow{\prod_{\ell=\alpha_{n,\mathbf{i},\mathbf{j}} +1-\delta_{j_k,n}}^{\beta_{n,\mathbf{i},\mathbf{j}}+1-\delta_{j_k,n}}}\big(1^{\otimes m}\otimes\Delta\otimes 1^{\otimes \ell}\big)
\]
to the commutation relation (\ref{commutation}) for
\[
A = \{2-\delta_{i_1,1}, 4-\delta_{i_1,1},\dots,2k-\delta_{i_1,1}\}, \qquad B = [1;2k+1-\delta_{i_1,1}-\delta_{j_k,n}],
\]
which follows from Lemma~\ref{Lemmas 1.3.9 to 1.3.12}.
\end{proof}

It can easily be checked explicitly that one has
\begin{alignat}{4}
&[\Lambda_{\{1,3,4\}},\Lambda_{\{1,3\}}]= 0, \qquad &&[\Lambda_{\{1,3,4\}},\Lambda_{\{1,4\}}]= 0,\qquad &&[\Lambda_{\{1,3,4,5\}},\Lambda_{\{1,4\}}]= 0,&\nonumber\\ &[\Lambda_{\{1,2,4\}},\Lambda_{\{1,4\}}]= 0, \qquad && [\Lambda_{\{1,2,4,5\}},\Lambda_{\{1,4\}}]= 0, \qquad &&[\Lambda_{\{1,2,4,5\}},\Lambda_{\{1,5\}}]= 0,&\nonumber\\ &[\Lambda_{\{1,2,4,5,6\}},\Lambda_{\{1,5\}}]= 0, \qquad && [\Lambda_{\{1,2,4\}},\Lambda_{\{2,4\}}]= 0, \qquad& &[\Lambda_{\{1,2,4,5\}},\Lambda_{\{2,4\}}]= 0,& \nonumber\\
&[\Lambda_{\{1,2,4,5\}},\Lambda_{\{2,5\}}]= 0, \qquad &&[\Lambda_{\{1,2,4,5,6\}},\Lambda_{\{2,5\}}]= 0, \qquad &&[\Lambda_{\{1,2,3,5\}},\Lambda_{\{2,5\}}]= 0,&\nonumber \\ &[\Lambda_{\{1,2,3,5,6\}},\Lambda_{\{2,5\}}]= 0, \qquad &&[\Lambda_{\{1,2,3,5,6\}},\Lambda_{\{2,6\}}]= 0, \qquad &&[\Lambda_{\{1,2,3,5,6,7\}},\Lambda_{\{2,6\}}]= 0.& \label{explicit relations}
\end{alignat}
These relations are needed in order to show the following result.

\begin{Lemma}\label{lemma two intervals two-subset}Let $A = [1;j_1]\cup[i_2;j_2]$ and $B$ be such that $B\subseteq A$ and $\vert B\vert = 2$, then $\Lambda_A$ and $\Lambda_B$ commute.
\end{Lemma}
\begin{proof}[Sketch of the proof] If $B\subseteq [1;j_1]$, then by acting with morphisms $1^{\otimes \ell}\otimes\tau_R$ and $1^{\otimes m}\otimes \Delta$, the claim follows from $[\Lambda_{[1;j_1+1]},\Lambda_{B}] = 0$, as asserted by Lemma~\ref{Proposition 1.5.2}. Similarly if $B\subseteq [i_2;j_2]$.

Now suppose $B\nsubseteq [1;j_1]$ and $B\nsubseteq [i_2;j_2]$, then $B\cap [1;j_1]$ consists of a single element $n_1$ and $B\cap [i_2;j_2]$ of a single element $n_2$. One can distinguish 16 cases: each of the 4 cases $n_1 = 1 = j_1$, $n_1 = 1 \neq j_1$, $n_1 = j_1 \neq 1$, $j_1\neq n_1\neq 1$ can be combined with each of the 4 cases determined similarly by the mutual equality of $n_2$, $j_2$ and $i_2$. Each of those cases follows from one of the relations in~(\ref{explicit relations}) upon applying $\Delta$ on suitable tensor product positions.
\end{proof}

The following lemma aims to remove part of the restrictions on the set $B$.

\begin{Lemma}\label{Lemma 1.5.5}Let $A = [1;j_1]\cup[i_2;j_2]$, and $B$ be such that $B\subseteq A$ and $\vert B\cap [i_2;j_2]\vert = 1$, then $\Lambda_A$ and $\Lambda_B$ commute.
\end{Lemma}
\begin{proof}By induction on $k = \vert B\cap[1;j_1]\vert$, the case $k = 0$ being trivial by Lemma \ref{lemma Gamma of singleton} and $k = 1$ following from Lemma \ref{lemma two intervals two-subset}. So let $k\geq 2$ and suppose the claim holds for all sets $B'\subset A$ with $\vert B'\cap [1;j_1]\vert$ strictly less than $k$.

If $B\cap [1;j_1]$ is a set of consecutive integers, the statement follows from Lemma \ref{lemma two intervals two-subset} upon applying $\Delta$ on suitable positions as in Proposition \ref{prop derived}. If not, then $B\cap [1;j_1]$ contains at least one hole, say between the elements $x_1$ and $x_2$. Let us write $B_1 = \{b\in B\colon b < x_1\}$ and $B_2 = \{b\in B\colon b > x_2\}$. By Lemma~\ref{Corollary 1.3.4} we may write
\begin{gather}\label{Gamma_B^q def}
\Lambda_B = \frac{[\Lambda_{B_1\cup [x_1;x_2-1]}, \Lambda_{[x_1+1;x_2]\cup B_2}]_q}{q^{-2}-q^2} + \frac{\Lambda_{[x_1+1;x_2-1]}\Lambda_{B_1\cup[x_1;x_2]\cup B_2} + \Lambda_{B_1\cup\{x_1\}}\Lambda_{\{x_2\}\cup B_2}}{q+q^{-1}}\!\!\!
\end{gather}
and hence it suffices to show that $\Lambda_A$ commutes with each of the terms in the right-hand side. For $\Lambda_{B_1\cup[x_1;x_2-1]}$, $\Lambda_{[x_1+1;x_2-1]}$ and $\Lambda_{B_1\cup\{x_1\}}$ this follows from Lemma~\ref{Proposition 1.5.2} with $A = [1;j_1+1]$, upon applying repeatedly $1^{\otimes m}\otimes\tau_R$ and $1^{\otimes\ell}\otimes\Delta$. For $\Lambda_{[x_1+1;x_2]\cup B_2}$, $\Lambda_{B_1\cup [x_1;x_2]\cup B_2}$ and $\Lambda_{\{x_2\}\cup B_2}$ this follows from the induction hypothesis. This concludes the proof.
\end{proof}

We can now also remove the constraint on $\vert B\cap [i_2;j_2]\vert$. The proof is completely similar to the previous one, we will hence only sketch it.

\begin{Proposition}\label{Proposition 1.5.6}Let $A = [1;j_1]\cup[i_2;j_2]$ and $B$ be such that $B\subseteq A$, then $\Lambda_A$ commutes with~$\Lambda_B$.
\end{Proposition}
\begin{proof}[Sketch of the proof] If $B\cap [1;j_1]$ is empty, the result follows from Lemma \ref{Proposition 1.5.2}. If it is not, then we will proceed by induction on $k = \vert B\cap[i_2;j_2]\vert$. The case $k = 0$ follows from Lemma~\ref{Proposition 1.5.2} and $k = 1$ from Lemma~\ref{Lemma 1.5.5}. Let hence $k\geq 2$ and suppose the claim holds for all values of $\vert B\cap [i_2;j_2]\vert$ strictly less than $k$. We again distinguish between $B\cap [i_2;j_2]$ without holes, in which case Lemma~\ref{Lemma 1.5.5} suffices, and $B\cap [i_2;j_2]$ containing holes, in which case we can rewrite~$\Lambda_B$ using~(\ref{Gamma_B^q def}) and apply Lemma~\ref{Proposition 1.5.2} and the induction hypothesis.
\end{proof}

Let us now take a look at the case where the set $A$ consists of more than just 2 discrete intervals.

\begin{Lemma}\label{Lemma 1.5.7 and Corollary 1.5.8}
Let $A = [1;j_1]\cup[i_2;j_2]\cup\dots\cup[i_{k-1};j_{k-1}]\cup[i_k;j_k]$, with $k\geq 3$. Then $\Lambda_A$ commutes with $\Lambda_B$, where $B$ is any of the sets
\begin{alignat*}{3}
& B= [i_1';j_1]\cup[i_2;j_2]\cup\dots\cup[i_{k-1};j_{k-1}]\cup[i_k;j_k'], \quad \ && B = [i_1';j_1]\cup[i_2;j_2]\cup\dots\cup[i_{k-1};j_{k-1}], &\\
& B = [i_2;j_2]\cup\dots\cup[i_{k-1};j_{k-1}]\cup[i_k;j_k'], \quad && B= [i_2;j_2]\cup\dots\cup[i_{k-1};j_{k-1}],&
\end{alignat*}
where $1\leq i_1'\leq j_1$ and $i_k\leq j_k'\leq j_k$.
\end{Lemma}
\begin{proof}Let us start with the first claim. If $i_1' = 1$ or $j_k' = j_k$, this follows from Lemma~\ref{Lemma 1.3.7}. So suppose $1 < i_1' , j_k' < j_k$. Then the statement follows upon applying
\begin{gather*}
\chi = \left(\overrightarrow{\prod_{\ell=j_k-i_1'+1}^{j_k-2}}\big(\Delta\otimes 1^{\otimes \ell}\big)\right)\left(\overrightarrow{\prod_{\ell=j_k-j_1}^{j_k-i_1'-1}}\big(1\otimes\Delta\otimes 1^{\otimes \ell}\big)\right)
\left(\overleftarrow{\prod_{n=1}^{2k-3}}\overrightarrow{\prod_{\ell=\alpha_{n,\mathbf{i},\mathbf{j}}}^{\beta_{n,\mathbf{i},\mathbf{j}}}}\big(1^{\otimes (n+1)}\otimes\Delta\otimes 1^{\otimes \ell}\big)\right)\\
\hphantom{\chi =}{}\times
\left(\overrightarrow{\prod_{\ell=j_k-j_k'}^{j_k-i_k-1}}\big(1^{\otimes (2k-1)}\otimes\Delta\otimes 1^{\otimes \ell}\big)\right)\left(\overrightarrow{\prod_{\ell=0}^{j_k-j_k'-2}}\big(1^{\otimes 2k}\otimes\Delta\otimes 1^{\otimes \ell}\big)\right)
\end{gather*}
to the statement of Lemma~\ref{Corollary 1.4.10}, as asserted by Proposition~\ref{prop derived}. The second claim follows from the first with $A$ replaced by $[1;j_1]\cup[i_2;j_2]\cup\dots\cup[i_{k-1};j_{k-1}+1]$, acted upon by iterations of $1^{\otimes \ell}\otimes\tau_R$ and $1^{\otimes m}\otimes \Delta$. Similarly for the other statements.
\end{proof}

For any $m\in\N$ one can define the following operators:
\begin{gather}\label{chi extension left}
\chi_{m} = \overleftarrow{\prod_{n=1}^{m}}\left(\left(\overrightarrow{\prod_{\ell=j_k-j_{n}}^{j_k-i_n-\delta_{n,1}}}\big(\Delta\otimes 1^{\otimes \ell}\big)\right)\left(\overrightarrow{\prod_{\ell=j_k-i_{n+1}+1}^{j_k-j_n-1}}\big(\tau_L\otimes 1^{\otimes \ell}\big)\right)\right), \\
\label{chi extension right}
\widetilde{\chi}_{m} = \overrightarrow{\prod_{n=m}^{k}}\left(\left(\overrightarrow{\prod_{\ell=i_{n}-1}^{j_n-1-\delta_{n,k}}}\big(1^{\otimes \ell}\otimes\Delta\big)\right)\left(\overrightarrow{\prod_{\ell=j_{n-1}}^{i_n-2}}\big(1^{\otimes \ell}\otimes\tau_R\big)\right)\right).
\end{gather}
These will be of use to show the following commutation relation.

\begin{Lemma}\label{Lemma 1.5.9, Corollary 1.5.10 en 1.5.11}Let $A = [1;j_1]\cup[i_2;j_2]\cup\dots\cup[i_{k-1};j_{k-1}]\cup[i_k;j_k]$ and let $B$ be any of the sets
\[
B = [i_1';j_1]\cup[i_2;j_2]\cup\dots\cup[i_{k-1};j_{k-1}]\cup B_k, \qquad B = [i_2;j_2]\cup\dots\cup[i_{k-1};j_{k-1}]\cup B_k,
\]
where $k\geq 3$, $1\leq i_1'\leq j_1$ and where $B_k\subseteq [i_k;j_k]$. Then $\Lambda_A$ and $\Lambda_B$ commute.
\end{Lemma}
\begin{proof}We will start with the first claim. If $B_k$ is empty or of the form $[i_k;j_k']$ with $i_k\leq j_k'\leq j_k$, the statement follows from Lemma~\ref{Lemma 1.5.7 and Corollary 1.5.8}. So suppose $B_k$ is nonempty and not of that particular form. We will prove the special case $i_1' = j_1 = 2$, the general case then follows upon applying $\Delta$ on tensor product positions 1 and 2. Let $x_1 = \min([i_k;j_k]\setminus B_k)$. Define the sets
\begin{gather*}
B_1 = \{b\in B\colon b < x_1\} = \{2\}\cup[i_2;j_2]\cup\dots\cup[i_{k-1};j_{k-1}]\cup[i_k;x_1-1], \\ B_2 = \{b\in B_k\colon b > x_1\},
\end{gather*}
where we interpret $[i_k;x_1-1]$ as the empty set in case $x_1 = i_k$. Note that both sets are nonempty, by the assumptions on $B_k$. By Lemma~\ref{Corollary 1.3.4} we have
\begin{gather}\label{Gamma_B^q def again}
\Lambda_B = \frac{[\Lambda_{B_1\cup\{x_1\}},\Lambda_{\{x_1\}\cup B_2}]_q}{q^{-2}-q^2} + \frac{\Lambda_{\{x_1\}}\Lambda_{B\cup\{x_1\}} + \Lambda_{B_1}\Lambda_{B_2}}{q+q^{-1}}.
\end{gather}
Both $\Lambda_{B_1}$ and $\Lambda_{B_1\cup\{x_1\}}$ commute with $\Lambda_A$, as follows from Lemma~\ref{Lemma 1.5.7 and Corollary 1.5.8} and the form of~$B_1$. Moreover, Lemma~\ref{Proposition 1.5.2} asserts
\[
[\Lambda_{[1;j_k-i_k+2]},\Lambda_{B_2-(i_k-2)}] = 0, \qquad [\Lambda_{[1;j_k-i_k+2]},\Lambda_{(\{x_1\}\cup B_2)-(i_k-2)}] = 0.
\]
Acted upon with $\chi_{k-1}$, defined in~(\ref{chi extension left}), this implies $\Lambda_A$ commutes with $\Lambda_{B_2}$ and $\Lambda_{\{x_1\}\cup B_2}$. Hence it follows from~(\ref{Gamma_B^q def again}) that
\[
[\Lambda_A,\Lambda_{B}] = \frac{1}{q+q^{-1}} \Lambda_{\{x_1\}}[\Lambda_A,\Lambda_{B\cup\{x_1\}}].
\]
We now repeat our reasoning. Either $B_k\cup\{x_1\}$ is of the form $[i_k;j_k']$, in which case the statement follows from Lemma \ref{Lemma 1.5.7 and Corollary 1.5.8}. If not, then upon defining $x_2 = \min([i_k;j_k]\setminus(B_k\cup\{x_1\}))$, we find by the same arguments
\[
[\Lambda_A,\Lambda_{B\cup\{x_1\}}] = \frac{1}{q+q^{-1}}\Lambda_{\{x_2\}}[\Lambda_A,\Lambda_{B\cup\{x_1,x_2\}}].
\]
If we continue this process, then at some point the set $B_k\cup \{x_1,x_2,\dots,x_m\}$ will inevitably be of the form $[i_k;j_k']$, with $j_k' = \max(B)$, namely when we have \emph{filled up all the holes} in the set~$B_k$. At this point we have{\samepage
\[
[\Lambda_A,\Lambda_B] = \frac{1}{\big(q+q^{-1}\big)^m}\left(\prod_{\ell=1}^{m}\Lambda_{\{x_{\ell}\}}\right)[\Lambda_A,\Lambda_{B\cup\{x_1,\dots,x_{m}\}}] = 0,
\]
where the last step uses Lemma~\ref{Lemma 1.5.7 and Corollary 1.5.8}. This shows the claim for the first given form of~$B$.}

For the second form, our first claim asserts
\[
[\Lambda_{([i_2-1;j_2]\cup\dots\cup[i_{k-1};j_{k-1}]\cup[i_k;j_k])-(i_2-2)},\Lambda_{([i_2;j_2]\cup\dots\cup[i_{k-1};j_{k-1}]\cup B_k)-(i_2-2)}] = 0,
\]
which yields the anticipated result by the left extension process.
\end{proof}

We can now also replace $[i_1';j_1]$ in the previous lemma by an arbitrary set. The proof will be completely parallel to the previous one.

\begin{Proposition}
\label{Lemma 1.5.12}
Let $A = [1;j_1]\cup[i_2;j_2]\cup\dots\cup[i_{k-1};j_{k-1}]\cup[i_k;j_k]$, with $k\geq 3$, and let $B$ be such that $B\subseteq A$ and $[i_2;j_2]\cup\dots\cup[i_{k-1};j_{k-1}]\subseteq B$, then $\Lambda_A$ commutes with $\Lambda_B$.
\end{Proposition}
\begin{proof}[Sketch of the proof] If $B\cap [1;j_1]$ is empty or of the form $[i_1';j_1]$, the claim follows from Lemma~\ref{Lemma 1.5.9, Corollary 1.5.10 en 1.5.11}. If not, then with $x_1 = \max([1;j_1]\setminus B)$, $B_1 = \{b\in B\colon b < x_1\}$ and $B_2 = \{b\in B\colon b > x_1\} = [x_1+1;j_1]\cup[i_2;j_2]\cup\dots\cup[i_{k-1};j_{k-1}]\cup (B\cap [i_k;j_k])$, we have again the relation~(\ref{Gamma_B^q def again}). Each of the terms in the right-hand side of (\ref{Gamma_B^q def again}) commutes with $\Lambda_A$, by Lemmas~\ref{Proposition 1.5.2} and~\ref{Lemma 1.5.9, Corollary 1.5.10 en 1.5.11}, except for $\Lambda_{B\cup\{x_1\}}$. Defining recursively
$x_i = \max([1;j_1]\setminus (B\cup\{x_1,\dots,x_{i-1}\}))$, we find
\[
[\Lambda_A,\Lambda_B] = \frac{1}{q+q^{-1}} \Lambda_{\{x_1\}}[\Lambda_A,\Lambda_{B\cup\{x_1\}}] = \frac{1}{\big(q+q^{-1}\big)^2}\Lambda_{\{x_{1}\}}\Lambda_{\{x_2\}}[\Lambda_A,\Lambda_{B\cup\{x_1,x_2\}}] = \cdots,
\]
which eventually becomes zero, since at some point the set $(B\cup\{x_1,\dots,x_m\})\cap[1;j_1]$ will inevitably be of the form $[i_1';j_1]$, with $i_1' = \min(B)$, such that Lemma \ref{Lemma 1.5.9, Corollary 1.5.10 en 1.5.11} will be applicable.
\end{proof}

\section{Proof of Theorem \ref{thm - commutation general}}\label{section - proof commutation general}
\begin{proof}Without loss of generality, we may assume $\min(A) = 1$ and thus we can write $A = [1;j_1]\cup[i_2;j_2]\cup\dots\cup[i_k;j_k]$. For ease of notation, we will write $A_{\ell}$ for the discrete interval $[i_{\ell};j_{\ell}]$. We will proceed by induction on $k$.

For $k = 1$ the statement follows from Lemma \ref{Proposition 1.5.2}. For $k = 2$ we may invoke Proposition \ref{Proposition 1.5.6}. Suppose hence that $k \geq 3$ and that the statement holds for any set $A'$ consisting of strictly less than $k$ discrete intervals, and for any set $B'$ contained in $A'$.
We distinguish four cases.

{\bf Case 1:} $A_2\cup\dots\cup A_{k-1}\subseteq B$.

This case follows immediately from Proposition \ref{Lemma 1.5.12}.

{\bf Case 2:} $B\cap A_k = \varnothing$.

In this case the induction hypothesis asserts $[\Lambda_{A_1\cup\dots\cup A_{k-1}\cup\{j_{k-1}+1\}},\Lambda_{B}] \allowbreak = 0$, which implies our claim upon acting repeatedly with $1^{\otimes\ell}\otimes\tau_R$ and $1^{\otimes m}\otimes\Delta$.

{\bf Case 3:} $B\cap A_1 = \varnothing$.

Similarly, using the induction hypothesis and the left extension process.

{\bf Case 4:} $A_2\cup\dots\cup A_{k-1} \nsubseteq B$, $B\cap A_1\neq\varnothing$ and $B\cap A_k\neq\varnothing$.

Let $x_1\in A\setminus B$ be such that $x_1\in A_i$ with $i\in \{2,\dots,k-1\}$. Let us define the sets $B_1 = \{b\in B\colon b < x_1\}$ and $B_2 = \{b\in B\colon b > x_1\}$. Then it follows from Lemma~\ref{Corollary 1.3.4} that
\begin{gather}\label{basis relation Gamma}
\Lambda_B = \frac{[\Lambda_{B_1\cup\{x_1\}},\Lambda_{\{x_1\}\cup B_2}]_q}{q^{-2}-q^2} + \frac{\Lambda_{\{x_1\}}\Lambda_{B\cup\{x_1\}} + \Lambda_{B_1}\Lambda_{B_2}}{q+q^{-1}}.
\end{gather}
Note that none of the indexing sets above is empty, by our assumptions on~$B$. Observe that $B_1$ and $B_1\cup\{x_1\}$ are contained in $A_1\cup\dots\cup A_i\cup\{j_i+1\}$, which contains strictly fewer discrete intervals than $k$. Hence the induction hypothesis asserts
\begin{gather*}
[\Lambda_{B_1},\Lambda_{A_1\cup\dots\cup A_i\cup\{j_i+1\}}] = 0, \qquad [\Lambda_{B_1\cup\{x_1\}},\Lambda_{A_1\cup\dots\cup A_i\cup\{j_i+1\}}] = 0,
\end{gather*}
which by $\widetilde{\chi}_{i+1}$, as defined in~(\ref{chi extension right}), asserts that $\Lambda_{B_1}$ and $\Lambda_{B_1\cup\{x_1\}}$ commute with $\Lambda_A$. A similar reasoning using the induction hypothesis and the morphism $\chi_{i-1}$, defined in~(\ref{chi extension left}), shows that~$\Lambda_A$ commutes with~$\Lambda_{B_2}$ and $\Lambda_{\{x_1\}\cup B_2}$. The expression~(\ref{basis relation Gamma}) now implies
\[
[\Lambda_A,\Lambda_B] = \frac{1}{q+q^{-1}}\Lambda_{\{x_1\}}[\Lambda_A,\Lambda_{B\cup\{x_1\}}].
\]

We may now repeat this reasoning. Either $B\cup\{x_1\}$ contains $A_2\cup\dots\cup A_{k-1}$, in which case $[\Lambda_A,\Lambda_{B\cup\{x_1\}}] = 0$ by Proposition~\ref{Lemma 1.5.12}. If not, then as established before $A\setminus B$ contains an element $x_2$ contained in $A_j$ for a certain $j\in\{2,\dots,k-1\}$, and as before this implies that
\[
[\Lambda_A,\Lambda_{B\cup\{x_1\}}] = \frac{1}{q+q^{-1}}\Lambda_{\{x_2\}}[\Lambda_A,\Lambda_{B\cup\{x_1,x_2\}}].
\]
If we continue this process, then at some point the set $B\cup\{x_1,\dots,x_m\}$ must inevitably contain $A_2\cup\dots\cup A_{k-1}$. At this point we have{\samepage
\[
[\Lambda_A,\Lambda_B] = \frac{1}{\big(q+q^{-1}\big)^m}\left(\prod_{\ell=1}^{m}\Lambda_{\{x_{\ell}\}}\right)[\Lambda_A,\Lambda_{B\cup\{x_1,\dots,x_{m}\}}] = 0,
\]
where in the last step we have used Proposition~\ref{Lemma 1.5.12}. This concludes the proof.}
\end{proof}

\section{Proof of Theorem \ref{thm - most general algebra relations}}\label{section - proof most general algebra relations}

\begin{proof}We will start by proving the case $A_1, A_2, A_3\neq\varnothing$ and $A_4 = \varnothing$. It suffices to do this in the situation where $A_1 = \{1\}$, $\min(A_2) = 2$ and $A_3 = \{\max(A_2)+1\}$, the general case then follows upon applying suitable morphisms of the form $\Delta\otimes 1^{\otimes n}$, $\tau_L\otimes 1^{\otimes n}$, $1^{\otimes n}\otimes\Delta$ and $1^{\otimes n}\otimes\tau_R$. Let us write $A_2$ in the form
\[
A_2 = [2;j_1]\cup[i_2;j_2]\cup\dots\cup[i_k;j_k].
\]
Cases (\ref{most general sets 1}), (\ref{most general sets 2}) and (\ref{most general sets 3}) then follow from (\ref{sets 1}), (\ref{sets 2}) and (\ref{sets 3}) respectively, by
\begin{gather*}
\chi =
\left(\overrightarrow{\prod_{\ell=j_k-j_1+1}^{j_k-2}}\big(1\otimes\Delta\otimes 1^{\otimes \ell}\big)\right)
\left(\overleftarrow{\prod_{n=1}^{2k-2}}\overrightarrow{\prod_{\ell=\alpha_{n,\mathbf{i},\mathbf{j}}+1}^{\beta_{n,\mathbf{i},\mathbf{j}}+1}}\big(1^{\otimes (n+1)}\otimes\Delta\otimes 1^{\otimes \ell}\big)\right).
\end{gather*}
Indeed, it follows from Proposition \ref{prop derived}, after separating the terms corresponding to $n = 0$ and applying coassociativity, that
\begin{gather*}
\Lambda_A = \left(\overrightarrow{\prod_{\ell=j_k-j_1+1}^{j_k-2}}\big(1\otimes\Delta\otimes 1^{\otimes \ell}\big)\right)\big(\Delta\otimes 1^{\otimes (j_k-j_1+1)}\big)\\
\hphantom{\Lambda_A =}{}\times \left(\overleftarrow{\prod_{n=1}^{2k-2}}\overrightarrow{\prod_{\ell=\alpha_{n,\mathbf{i},\mathbf{j}}+1}^{\beta_{n,\mathbf{i},\mathbf{j}}+1}}\big(1^{\otimes n}\otimes\Delta\otimes 1^{\otimes \ell}\big)\right)\Lambda_{\{1,3,5,\dots,2k-1\}},
\end{gather*}
where $\Lambda_{\{1,3,5,{\dots},2k{-}1\}}\!$ is considered inside a $(2k)$-fold tensor product.
The morphism $\Delta {\otimes} 1^{\!{\otimes} (j_k{-}j_1{+}1)\!}\!$ can be shifted through the subsequent ones by~(\ref{Lemma 1.2.1 eq}), thus acting directly on $\Lambda_{\{1,3,5,\dots,2k-1\}}$ by sending it to $\Lambda_{\{1,2,4,6,\dots,2k\}}$. Hence we have
$
\Lambda_A = \chi(\Lambda_{\{1,2,4,6,\dots,2k\}})
$.
Similarly, $\chi$ sends $\Lambda_{\{2,4,6,\dots,2k,2k+1\}}$ to $\Lambda_B$ and so on.

The case $A_1 = \varnothing$ and $A_2, A_3, A_4\neq\varnothing$ is identical.

Now consider the case $A_2 = \varnothing$ and $A_1, A_3, A_4 \neq \varnothing$. Then (\ref{most general sets 2}) and (\ref{most general sets 3}) follow from Theorem \ref{thm - commutation general} and (\ref{Lambda emptyset}). To show (\ref{most general sets 1}), it again suffices to prove $[\Lambda_A, \Lambda_B] = 0$, by (\ref{Lambda emptyset}), and to take $A_1 = \{1\}$, $\min(A_3) = 2$ and $A_4 = \{\max(A_3)+1\}$. Let us write $A_3$ in the form
\[
A_3 = [2;j_1]\cup[i_2;j_2]\cup\dots\cup[i_k;j_k].
\]
By
\[
\chi =\overleftarrow{\prod_{n=0}^{2k-2}}\overrightarrow{\prod_{\ell=\alpha_{n,\mathbf{i},\mathbf{j}}}^{\beta_{n,\mathbf{i},\mathbf{j}}}}\big(1^{\otimes (n+1)}\otimes\Delta\otimes 1^{\otimes (\ell+1)}\big)
\]
with $i_1 = 2$, this follows from Lemma \ref{Lemma 1.3.8}.

The case $A_3 = \varnothing$ and $A_1, A_2, A_4 \neq\varnothing$ is identical.

If two or more of the $A_i$ are empty, then again the statements follow from Theorem~\ref{thm - commutation general} and~(\ref{Lambda emptyset}).

Finally, suppose none of the $A_i$ is empty. Again, it suffices to consider the special case where $A_1 = \{1\}$, $\min(A_2) = 2$ and $A_4 = \{\max(A_3)+1\}$. We will write $A_2$ and $A_3$ as disjoint unions of discrete intervals. If $\min(A_3) = \max(A_2) + 1$, then
\begin{gather*}
A_2 = [2;j_1]\cup[i_2;j_2]\cup\dots\cup[i_k;j_k], \\ A_3 = [j_k+1;j_{k+1}]\cup[i_{k+2};j_{k+2}]\dots\cup[i_{k+\ell+1};j_{k+\ell+1}].
\end{gather*}
Cases (\ref{most general sets 1}), (\ref{most general sets 2}) and (\ref{most general sets 3}) then follow from (\ref{sets 4 without gap}), (\ref{sets 5 without gap}) and (\ref{sets 6 without gap}) respectively, by
\begin{gather*}
\chi = \left(\overleftarrow{\prod_{n=0}^{2k-2}} \overrightarrow{\prod_{\ell=\alpha_{n,\mathbf{i},\mathbf{j}}+1}^{\beta_{n,\mathbf{i},\mathbf{j}}+1}}\big(1^{\otimes (n+1)}\otimes\Delta\otimes 1^{\otimes \ell}\big)\right)\left(\overrightarrow{\prod_{\ell=j_{k+\ell+1}-j_{k+1}+1}^{j_{k+\ell+1}-j_k-1}}\big(1^{\otimes 2k}\otimes\Delta\otimes 1^{\otimes \ell}\big)\right)\\
\hphantom{\chi =}{}\times \left(\overleftarrow{\prod_{n=2k+1}^{2k+2\ell}} \overrightarrow{\prod_{\ell=\alpha_{n,\mathbf{i},\mathbf{j}}+1}^{\beta_{n,\mathbf{i},\mathbf{j}}+1}}\big(1^{\otimes n}\otimes\Delta\otimes 1^{\otimes \ell}\big)\right),
\end{gather*}
where $\mathbf{i} = (2,i_2,\dots,i_k, j_k+1, i_{k+2}, \dots, i_{k+\ell+1})$ and $\mathbf{j} = (j_1,\dots,j_{k+\ell+1})$.

If $\min(A_3) > \max(A_2)+1$, then we have
\begin{gather*}
A_2 = [2;j_1]\cup[i_2;j_2]\cup\dots\cup[i_k;j_k], \qquad A_3 = [i_{k+1};j_{k+1}]\cup[i_{k+2};j_{k+2}]\dots\cup[i_{k+\ell+1};j_{k+\ell+1}]
\end{gather*}
and by
\begin{gather*}
\chi = \overleftarrow{\prod_{n=0}^{2k+2\ell}} \overrightarrow{\prod_{\ell=\alpha_{n,\mathbf{i},\mathbf{j}}+1}^{\beta_{n,\mathbf{i},\mathbf{j}}+1}}\big(1^{\otimes (n+1)}\otimes\Delta\otimes 1^{\otimes \ell}\big),
\end{gather*}
with $\mathbf{i} = (2,i_2,\dots, i_{k+\ell+1})$ and $\mathbf{j} = (j_1,\dots,j_{k+\ell+1})$, the cases (\ref{most general sets 1}), (\ref{most general sets 2}) and (\ref{most general sets 3}) follow from (\ref{sets 4 with gap}), (\ref{sets 5 with gap}) and (\ref{sets 6 with gap}) respectively.
\end{proof}

\section[Higher rank $q$-Bannai--Ito relations]{Higher rank $\boldsymbol{q}$-Bannai--Ito relations}\label{section q-BI}

Each of the results and proofs outlined in this paper carries over to the higher rank $q$-Bannai--Ito algebra, which is canonically isomorphic to the higher rank Askey--Wilson algebra. In this final section we will apply the methods of Section~\ref{Section - Defining the higher rank generators} to construct this higher rank extension of the $q$-Bannai--Ito algebra and we will argue why the Theorems~\ref{thm - commutation general} and \ref{thm - most general algebra relations} have natural $q$-Bannai--Ito analogs.

Let us start by the defining the quantum superalgebra $\ospq$ as the $\mathbb{Z}_2$-graded associative algebra over a field $\K$ generated by elements $A_+$, $A_-$, $K$, $K^{-1}$ and $P$ subject to the relations
\begin{gather*}
KA_{+}K^{-1} = q^{1/2} A_{+}, \qquad KA_{-}K^{-1} = q^{-1/2} A_{-}, \qquad \{A_+,A_-\} = \frac{K^2-K^{-2}}{q^{1/2}-q^{-1/2}}, \\ \{P,A_{\pm}\} = 0, \qquad
 [P,K] = 0, \qquad [P,K^{-1}] = 0, \qquad KK^{-1} = K^{-1}K = 1, \qquad P^2 = 1.
\end{gather*}
A Casimir element is given by
\begin{gather*}
\Gamma^q = \left(-A_+A_-+\frac{q^{-1/2}K^2-q^{1/2}K^{-2}}{q-q^{-1}} \right)P.
\end{gather*}
Its coproduct $\Delta\colon \ospq\to\ospq^{\otimes 2}$ and counit $\epsilon\colon \ospq\to\K$ are defined by
\begin{alignat}{4}\label{def: Coproduct}
&\Delta(A_{\pm})= A_{\pm}\otimes KP + K^{-1}\otimes A_{\pm}, \quad \ \ \! &&\Delta\big(K^{\pm 1}\big)= K^{\pm 1}\otimes K^{\pm 1}, \quad \ \ \! &&\Delta(P)= P\otimes P,& \\
&\epsilon (A_{\pm})= 0, \qquad &&\epsilon\big(K^{\pm 1}\big)=1, \qquad &&\epsilon(P)= 1.&\nonumber
\end{alignat}
We will also need the definition of the $q$-anticommutator, given by
\[
\{X,Y\}_q = q^{1/2}XY+q^{-1/2}YX.
\]
Now let $\mathcal{I}_R$ be the subalgebra of $\ospq$ generated by $A_+K$, $A_-K$, $K^2P$ and $\Gamma^q$ and similarly write $\mathcal{I}_L$ for the $\ospq$-subalgebra generated by the elements $A_+K^{-1}P$, $A_-K^{-1}P$, $K^{-2}P$ and $\Gamma^q$. If one defines the algebra morphism $\tau_R\colon \mathcal{I}_R\to \ospq\otimes\mathcal{I}_R$ by
\begin{gather}
\tau_R(A_-K) = K^2P\otimes A_-K, \nonumber\\
\tau_R(A_+K) = \big(K^{-2}P\otimes A_+K\big)+q^{-1/2}\big(q-q^{-1}\big)\big(A_+^2P\otimes A_-K\big) \nonumber\\
\hphantom{\tau_R(A_+K) =}{} + q^{-1/2}\big(q^{1/2}-q^{-1/2}\big)\big(A_+K^{-1}P\otimes K^2P\big) +q^{-1/2}\big(q-q^{-1}\big)\big(A_+K^{-1}P\otimes \Gamma^q\big), \nonumber\\
\tau_R\big(K^2P\big) = 1\otimes K^2P - \big(q-q^{-1}\big)(A_+K\otimes A_-K),\nonumber\\
\tau_R\big(\Gamma^q\big) = 1\otimes \Gamma^q,\label{tau_R q-BI def}
\end{gather}
then $\mathcal{I}_R$ is readily checked to be a left coideal subalgebra of $\ospq$ and a left $\ospq$-comodule with coaction $\tau_R$. Similarly, the subalgebra $\mathcal{I}_L$ is a right coideal subalgebra of $\ospq$ and a right $\ospq$-comodule with coaction $\tau_L\colon \mathcal{I}_L\to\mathcal{I}_L\otimes \ospq$ defined by
\begin{gather}
\tau_L\big(A_-K^{-1}P\big) = A_-K^{-1}P\otimes K^{-2}P, \nonumber\\
\tau_L\big(A_+K^{-1}P\big) = A_+K^{-1}P\otimes K^2P - q^{1/2}\big(q-q^{-1}\big) A_-K^{-1}P\otimes A_+^2P \nonumber\\
\hphantom{\tau_L\big(A_+K^{-1}P\big) =}{} - q^{1/2}\big(q^{1/2}-q^{-1/2}\big)K^{-2}P\otimes A_+K - q^{1/2}\big(q-q^{-1}\big)\Gamma^q\otimes A_+K,\nonumber\\
\tau_L\big(K^{-2}P\big) = K^{-2}P\otimes 1 - \big(q-q^{-1}\big) A_-K^{-1}P\otimes A_+K^{-1}P, \nonumber\\
\tau_L\big(\Gamma^q\big) = \Gamma^q\otimes 1.\label{tau_L q-BI def}
\end{gather}
In analogy to Corollary \ref{cor cotensor product}, one can easily verify that the element $\Delta(\Gamma^q)\in\mathcal{I}_L\otimes\mathcal{I}_R$ lies in the cotensor product of $\mathcal{I}_L$ and $\mathcal{I}_R$, i.e.,
\[
(1\otimes\tau_R)\Delta\big(\Gamma^q\big) = (\tau_L\otimes 1)\Delta\big(\Gamma^q\big).
\]
This observation makes it possible to repeat the extension processes outlined in Section \ref{Section - Defining the higher rank generators} for the $q$-Bannai--Ito case, and in particular to state the following definition.

\begin{Definition}The $q$-Bannai--Ito algebra of rank $n-2$ is the subalgebra of $\ospq^{\otimes n}$ ge\-ne\-rated by the elements $\Gamma_A^q$ with $A\subseteq [1;n]$, which are constructed by the left, right or mixed extension processes of Definitions \ref{right extension process def}, \ref{left extension process def} and \ref{mixed extension process def}, upon replacing $\Lambda\in\Uqsl$ by $\Gamma^q\in\ospq$ and $\Delta$, $\tau_R$ and $\tau_L$ by the morphisms (\ref{def: Coproduct}), (\ref{tau_R q-BI def}) and (\ref{tau_L q-BI def}).
\end{Definition}

It was explained in \cite[Section 2.3]{DeBie&DeClercq&vandeVijver-2018} that the rank $n-2$ $q$-Bannai--Ito algebra is isomorphic to the Askey--Wilson algebra of the same rank. The isomorphism is explicitly given by
\[
\Lambda_A \mapsto -i\big(q-q^{-1}\big)\Gamma_A^q, \qquad q\mapsto iq^{1/2},
\]
where $i\in\K$ is a square root of $-1$.

Looking back at the Sections \ref{Section - Main results} to \ref{section - proof most general algebra relations}, it is clear that the Theorems \ref{thm - commutation general} and \ref{thm - most general algebra relations} were proven based solely on the coassociativity of the coproduct, the cotensor product property and the rank one relations (\ref{rank one 12-23})--(\ref{rank one 13-12}). The whole strategy of proof can thus be repeated to extend the rank one $q$-Bannai--Ito relations
\begin{gather*}
\big\{\Gamma_{\{1,2\}}^q,\Gamma_{\{2,3\}}^q\big\}_q = \Gamma_{\{1,3\}}^q + \big(q^{1/2}+q^{-1/2}\big)\big( \Gamma_{\{2\}}^q\Gamma_{\{1,2,3\}}^q + \Gamma_{\{1\}}^q\Gamma_{\{3\}}^q \big), \\
\big\{\Gamma_{\{2,3\}}^q,\Gamma_{\{1,3\}}^q\big\}_q = \Gamma_{\{1,2\}}^q + \big(q^{1/2}+q^{-1/2}\big)\big( \Gamma_{\{3\}}^q\Gamma_{\{1,2,3\}}^q + \Gamma_{\{1\}}^q\Gamma_{\{2\}}^q \big), \\
\big\{\Gamma_{\{1,3\}}^q,\Gamma_{\{1,2\}}^q\big\}_q = \Gamma_{\{2,3\}}^q + \big(q^{1/2}+q^{-1/2}\big)\big( \Gamma_{\{1\}}^q\Gamma_{\{1,2,3\}}^q + \Gamma_{\{2\}}^q\Gamma_{\{3\}}^q \big),
\end{gather*}
as shown in \cite{Genest&Vinet&Zhedanov-2016}, to higher rank. This leads to the following analog of Theorems \ref{thm - commutation general} and \ref{thm - most general algebra relations}.

\begin{Theorem}Let $A,B\subseteq [1;n]$ be such that $B\subseteq A$, then $\Gamma_A^q$ and $\Gamma_B^q$ commute.
\end{Theorem}

\begin{Theorem}Let $A_1$, $A_2$, $A_3$ and $A_4$ be $($potentially empty$)$ subsets of $[1;n]$ satisfying
\[
A_1\prec A_2\prec A_3\prec A_4,
\]
where we have used Definition~{\rm \ref{Notation prec}}. The relation
\[
\big\{\Gamma_{A}^q,\Gamma_{B}^q\big\}_q = \Gamma_{(A\cup B)\setminus(A\cap B)}^q + \big(q^{1/2}+q^{-1/2}\big)\big(\Gamma_{A\cap B}^q\Gamma_{A\cup B}^q+ \Gamma_{A\setminus(A\cap B)}^q\Gamma_{B\setminus(A\cap B)}^q \big)
\]
is satisfied for $A$ and $B$ defined by one of the relations \eqref{most general sets 1}--\eqref{most general sets 3}.
\end{Theorem}

\section{Conclusions and outlook}

In this paper, we have presented several construction techniques for the higher rank Askey--Wilson algebra ${\rm AW}(n)$, equivalent with the algorithm given in \cite{DeBie&DeClercq&vandeVijver-2018}. The key observation was the existence of a novel, right coideal comodule subalgebra of $\Uqsl$. We have proven a large class of algebraic identities inside ${\rm AW}(n)$, culminating in Theorems~\ref{thm - commutation general} and~\ref{thm - most general algebra relations}, by elementary and intrinsic methods. Each of the proofs also applies, mutatis mutandis, to the higher rank $q$-Bannai--Ito algebra, which is isomorphic to ${\rm AW}(n)$.

At present, it is still an open question whether the relations obtained in Theorems~\ref{thm - commutation general} and~\ref{thm - most general algebra relations} define the algebras ${\rm AW}(n)$ abstractly. Using computer algebra packages, we have obtained all subsets $A\subseteq \{1,\dots,n\}$ for which the relation (\ref{standard relation}) is satisfied, for several values of $n$, and each of them turned out to be of the form (\ref{most general sets 1}), (\ref{most general sets 2}) or (\ref{most general sets 3}). Hence if supplementary relations are needed to obtain a complete set of defining relations for the algebras ${\rm AW}(n)$, they will most likely not be of the form (\ref{standard relation}). Other related problems are the construction of similar algebras for more general quantum groups $U_q(\mathfrak{g})$ and the behavior of the obtained algebras at $q$ a root of unity different from 1. We plan to look into these questions in further research.

\subsection*{Acknowledgements}
HDC is a PhD Fellow of the Research Foundation Flanders (FWO). This work was also supported by FWO Grant EOS 30889451. The author wishes to thank the anonymous referees for their valuable suggestions and comments.

\pdfbookmark[1]{References}{ref}
\LastPageEnding

\end{document}